\documentclass[12pt]{amsart}
\usepackage{dsfont}
\usepackage{amsfonts}
\usepackage{amsthm}
\usepackage{amssymb}
\usepackage{mathtools}
\usepackage{amsmath}
\usepackage{mathrsfs}
\usepackage{mathrsfs}
\usepackage{stmaryrd}
\usepackage{framed}
\usepackage{nccmath}
\usepackage{wrapfig}
\usepackage{enumitem}
\usepackage{geometry}
\usepackage{tikz,tikz-cd}
\usepackage{bm}
\usepackage[hidelinks]{hyperref}
\usepackage[stable]{footmisc}
\usepackage[new]{old-arrows}
\usepackage{draftwatermark}
\usepackage{quiver}

\usepackage{amsrefs}

\usepackage{lipsum}

\newcommand\blfootnote[1]{%
	\begingroup
	\renewcommand\thefootnote{}\footnote{#1}%
	\addtocounter{footnote}{-1}%
	\endgroup
}

\SetWatermarkText{DRAFT - GSN - 2022}
\SetWatermarkColor[gray]{0.95}
\SetWatermarkScale{0.4}
\NewDocumentCommand{\tens}{t_}
{%
	\IfBooleanTF{#1}
	{\tensop}
	{\otimes}%
}
\NewDocumentCommand{\tensop}{m}
{%
	\mathbin{\mathop{\otimes}\displaylimits_{#1}}%
}

\DeclareMathOperator{\Spec}{Spec}

\newcommand{\prtt}[1]{\left( #1 \right)}
\newcommand{\tA}[2]{#1 \tens_{R} #2}

\newcommand{\bN}{{\mathbb{N}}}

\newcommand{\sub}{\subseteq}
\newcommand{\ten}{\otimes}

\newcommand{\fq}{\mathfrak{q}}

\newcommand{\id}{\mbox{id}}

\newcommand{\spec}{\mbox{Spec}}
\newcommand{\Frac}{\mbox{Frac}}

\newcommand{\ideal}[1]{\left\langle #1 \right\rangle}

\theoremstyle{plain}

\newtheorem{defi}{Definition}[section]
\numberwithin{defi}{section}

\newtheorem{prop}[defi]{Proposition}
\newtheorem{teo}[defi]{Theorem}
\newtheorem{cor}[defi]{Corollary}
\newtheorem{lema}[defi]{Lemma}
\newtheorem{example}[defi]{Example}

\begin{document}
	
	\title{The contraction property on the relative weak normalization and Lipschitz saturation of algebras }
	\author{Thiago da Silva}
	\date{}
	
	\maketitle
	
	\begin{abstract}
		Inspired by the results obtained in \cite{SR}, in this work, we develop techniques to handle the contraction property for the relative weak normalization and Lipschitz saturation of algebras for the following types of algebras: universally injective, integral, radicial, and unramified. \blfootnote{2020 \textit{Mathematics Subjects Classification} 13B22, 14B07
			
			}
	\end{abstract}

	\tableofcontents

	\section*{Introduction}
	
	The notion of weak normalization was defined by Andreotti and Bombieri in \cite{AB} in the context of algebraic varieties over an algebraically closed field $k$. In \cite{Ma1}, Maranesi dealt with this structure in a language closer to commutative algebra, while investigating issues related to seminormalization, developed in \cite{GT}. 
	
	For an integral extension of rings $A\sub B$, Maranesi \cite{Ma1} defined $$\widetilde{A}_B:=\{b\in B\mid b\ten_A1-1\ten_Ab\in\sqrt{\ideal{0_{B\ten_AB}}}\},\footnote{Originally, Maranesi used the notation $\widehat{A}_B$.}$$
	
	\noindent and in \cite[Theorem I.6]{Ma}, she demonstrated this definition agrees with the weak normalization introduced by Andreotti and Bombieri in \cite{AB}. 
	
	Supposing that $A\sub B$ is an integral extension of $k$-algebras over a field $k$, in \cite{adkins}, Adkins extended Maranesi's definition of weak normalization of $A$ in $B$ over $k$ as $$\widetilde{A}_{B,k}:=\{b\in B\mid b\ten_k1-1\ten_kb\in\sqrt{\ker\varphi}\},\footnote{In his work, Adkins used the notation $A^*_B$.}$$
	
	\noindent where $\varphi:B\ten_kB\rightarrow B\ten_AB$ is the canonical map. 
	
	Meanwhile, in \cite{PT}, in the case of complex analytic algebras, Pham and Teissier defined the Lipschitz saturation of the local ring of a complex analytic variety on its normalization. In \cite{L}, Lipman extended this definition of Lipschitz saturation for a sequence of ring morphisms $R\rightarrow A\rightarrow B$ and defined what he called the \textit{relative Lipschitz saturation of $A$ in $B$}, denoted by $A^*_{B, R}$. In that work, Lipman developed several properties concerning this structure. A special one is the property of contraction. For a diagram of ring morphisms \[\begin{tikzcd}
		R & A & B \\
		{R'} & {A'} & {B'}
		\arrow[from=1-1, to=1-2]
		\arrow[from=1-2, to=1-3]
		\arrow[from=2-1, to=2-2]
		\arrow[from=2-2, to=2-3]
		\arrow[from=1-1, to=2-1]
		\arrow["{f}", from=1-3, to=2-3]
		\arrow[from=1-2, to=2-2]
	\end{tikzcd}\]
	
	\noindent in \cite[Property 7, p. 794]{L} Lipman showed that $A^*_{B,R}\sub f^{-1}((A')^*_{B',R'})$ and the equality holds if: 
	
	\begin{itemize}
		\item $R'=R$;
		
		\item $A=A'$;
		
		\item $f$ is faithfully flat or $f$ is integral with $\ker f$ nil ideal.
	\end{itemize}
	
	In \cite{SR}, the authors generalized the above results showing that the above diagram contracts working under conditions related to universal injectivity and integrality of some of the involved vertical morphisms.
	
	Inspired by what Lipman did for the Lipschitz saturation in \cite{L}, we will extend the definition of weak normalization for an arbitrary sequence of ring morphisms $R\rightarrow A\rightarrow B$, and we will call it \textit{relative weak normalization}.
	
	Having the extended definition of weak normalization at hand, the main objective of this work is to adapt the techniques used in \cite{SR} to obtain similar results for the weak normalization of $A$ in $B$ in the scenario proposed by Lipman, i.e., for a sequence of ring morphisms $R\rightarrow A\rightarrow B$, where $R$ is an arbitrary ring and $A\rightarrow B$ is not necessarily an inclusion.\footnote{In this work, the term \textit{ring} means a commutative ring with unit.}
	
	This work is organized as follows. In Section 1, we will review some definitions and known results; we also prove some auxiliary results for the construction of the main theorems and we introduce the Maranesi diagrams. In Section 2, we will show that the condition of universal injectivity together with \textit{radiciality} will allow us to obtain the contraction property in Maranesi diagrams in Theorem \ref{202410241447}. In Section 3, we investigate the contraction property for the integrality condition in Theorem \ref{202410241734}, and from this theorem we will conclude that weak normalization commutes with the quotient by an ideal. Finally, in Section 4 we prove some contraction properties for Maranesi and Lipman diagrams in the case of radicial and unramified algebras in Theorems \ref{202410262329},  \ref{202410271947} and \ref{202410280007}.

	\section{Preliminaries}

	In this section, we recall some basic definitions and results concerning the relative Lipschitz saturation, and we introduce the relative weak normalization of algebras. Let
	
	\begin{align*}
		R \overset{\tau}{\longrightarrow} A \overset{g}{\longrightarrow} B.
	\end{align*}
	\noindent be a sequence of ring morphisms, which induces algebra structures on $A$ and $B$. The \textit{diagonal map} from $B$ to its tensor product with itself over $R$ is defined as 
	\begin{align*}
		\Delta : B &\longrightarrow \tA{B}{B} \\
		b &\longmapsto b\ten_R1_B-1_B\ten_Rb.
	\end{align*}

	It is straightforward to check the following properties of $\Delta$:
	
	\begin{enumerate}
		\item $\Delta(b_1+b_2)=\Delta(b_1)+\Delta(b_2), \forall b_1,b_2\in B$;
		
		\item $\Delta(rb)=r\Delta(b), \forall r\in R$ and $b\in B$;
		
		\item (Leibniz rule) $\Delta(b_1b_2)=(b_1\ten_R 1)\Delta(b_2)+(1\ten_Rb_2)\Delta(b_1), \forall b_1,b_2\in B$. 
	\end{enumerate}
	
	The universal property of the tensor product ensures the existence of a unique $R$-algebras morphism $$\varphi: \tA{B}{B} \longrightarrow B\tens_{A}B$$
	
	\noindent which maps $x\otimes_R y\mapsto x\otimes_A y$, for all $x,y\in B$, that is known as the \textit{canonical morphism}. Besides, in  \cite[8.7]{kleiman}one can see that $$\ker\varphi=\langle ax\ten_R y-x\ten_Ray\mid a\in A\mbox{ and }x,y\in B \rangle.$$
	
	Notice that $ax\ten_R y-x\ten_Ray=(x\ten_Ry)(g(a)\ten_R1-1\ten_Rg(a))$. Therefore,$$\ker\varphi=\langle g(a)\ten_R1-1\ten_Rg(a)\mid a\in A\rangle=\Delta(g(A))(B\ten_RB),$$
	
	Thus, $\ker\varphi$ is the ideal of $B\ten_RB$ generated by the image of $\Delta\circ g$. Let us recall the definition of the relative Lipschitz saturation of algebras.

	\begin{defi}\cite[Def. 1.2]{L}
		The Lipschitz saturation of $A$ in $B$ relative to $R \overset{\tau}{\rightarrow} A \overset{g}{\rightarrow} B$ is the set
		\begin{align*}
			A^*_{B,R} := \left\{x \in B \mid \Delta(x) \in \overline{\ker \varphi}\right\}=\Delta^{-1}(\overline{\ker\varphi}).
		\end{align*}
	\end{defi}
	If $A^*_{B,R} = g(A)$ then $A$ is said to be saturated in $B$. Let us review some fundamental facts regarding relative Lipschitz saturation.
	
	\begin{prop}\cite[Property 1]{L}\label{prop_satsubanel}
		$A^*_{B,R}$ is an $R$-subalgebra of $B$ which contains $g(A)$.
	\end{prop}

	\begin{prop}\cite[Property 2.(ii)]{L}\label{202306222352}
		Let $A$ be an $R$-subalgebra of $B$ and let $C$ be an $R$-subalgebra of $B$ containing $A$ as an $R$-subalgebra. \[\begin{tikzcd}
			R & A & C & B
			\arrow["\tau", from=1-1, to=1-2]
			\arrow[hook, from=1-2, to=1-3]
			\arrow[hook, from=1-3, to=1-4]
			\arrow["g"', curve={height=12pt}, hook, from=1-2, to=1-4]
		\end{tikzcd}\]

		Then $A^*_{B,R} \subseteq C^*_{B,R}$.
	\end{prop}

	\begin{cor}\cite[Property 2.(i),(iii)]{L}\label{202306230013}
		If $A$ is an $R$-subalgebra of $B$ then: 
		
		\begin{enumerate}
			\item $A \subseteq A^*_{B,R}$.
			
			\item $(A^*_{B,R})^*_{B,R} = A^*_{B,R}$.
		\end{enumerate}
	\end{cor}

	Now we present the definition of the relative weak normalization concerning an arbitrary ring. For this definition, we follow the ideas that inspired Lipman in the case of relative Lipschitz saturation, and we will not require \textit{a priori} that $A$ is a subring of $B$, nor that the base ring $R$ is necessarily a field, as in \cite{adkins}.
	
	\begin{defi}
		The relative $R$-weak normalization of $A$ in $B$ is the set
		\begin{align*}
			\widetilde{A}_{B,R} := \left\{x \in B \mid \Delta(x) \in \sqrt{\ker \varphi}\right\}=\Delta^{-1}(\sqrt{\ker\varphi}).
		\end{align*}
	\end{defi}
	
	If $\widetilde{A}_{B,R}=g(A)$ then $A$ is said to be \textit{$R$-weakly normal} in $B$.
	
	Here we emphasize that we are using the symbol ``$\ast$'' for the Lipschitz saturation following the notation used by Lipman in \cite{L} for this structure. Thus, the reader may be aware that other works use the opposite choice that we did for these structures. It is easy to see that the relative weak normalization satisfies the following properties:
	
	\begin{itemize}
		\item $\widetilde{A}_{B,R}$ is an $R$-subalgebra of $B$ which contains $g(A)$;
		
		\item $A^*_{B,R}\sub\widetilde{A}_{B,R}$.
	\end{itemize}

	Let us establish some notation, which we will use in this work. Suppose that \[\begin{tikzcd}
		R & A & B \\
		{R'} & {A'} & {B'}
		\arrow["\tau", from=1-1, to=1-2]
		\arrow["g", from=1-2, to=1-3]
		\arrow["{\tau'}", from=2-1, to=2-2]
		\arrow["{g'}", from=2-2, to=2-3]
		\arrow["{f_R}"', from=1-1, to=2-1]
		\arrow["{f}", from=1-3, to=2-3]
		\arrow["{f_A}", from=1-2, to=2-2]
	\end{tikzcd}\]
	
	\noindent is a commutative diagram of ring morphisms. Then we can consider the diagram 
	% https://q.uiver.app/#q=WzAsMTEsWzAsMCwiUiJdLFsxLDAsIkEiXSxbMiwwLCJCIl0sWzQsMCwiQlxcdW5kZXJzZXR7Un17XFxvdGltZXN9QiJdLFs2LDAsIkJcXHVuZGVyc2V0e0F9e1xcb3RpbWVzfUIiXSxbMCwyLCJSJyJdLFsxLDIsIkEnIl0sWzIsMiwiQiciXSxbNCwyLCJCJ1xcdW5kZXJzZXR7Uid9e1xcb3RpbWVzfUInIl0sWzYsMiwiQidcXHVuZGVyc2V0e0EnfXtcXG90aW1lc31CJyJdLFs0LDEsIkInXFx1bmRlcnNldHtSfXtcXG90aW1lc31CJyJdLFswLDEsIlxcdGF1Il0sWzEsMiwiZyJdLFsyLDMsIlxcRGVsdGEiXSxbMyw0LCJcXHZhcnBoaSJdLFswLDUsImZfUiIsMl0sWzEsNiwiZl9BIiwyXSxbMiw3LCJmIl0sWzUsNiwiXFx0YXUnIiwyXSxbNiw3LCJnJyIsMl0sWzcsOCwiXFxEZWx0YSciLDJdLFszLDEwLCJmXFx1bmRlcnNldHtSfXtcXG90aW1lc31mIl0sWzEwLDgsInAiXSxbMyw4LCJcXGJhcntmfSIsMix7ImN1cnZlIjo1fV0sWzgsOSwiXFx2YXJwaGknIl1d
	\[\begin{tikzcd}
		R & A & B && {B\underset{R}{\otimes}B} && {B\underset{A}{\otimes}B} \\
		&&&& {B'\underset{R}{\otimes}B'} \\
		{R'} & {A'} & {B'} && {B'\underset{R'}{\otimes}B'} && {B'\underset{A'}{\otimes}B'}
		\arrow["\tau", from=1-1, to=1-2]
		\arrow["g", from=1-2, to=1-3]
		\arrow["\Delta", from=1-3, to=1-5]
		\arrow["\varphi", from=1-5, to=1-7]
		\arrow["{f_R}"', from=1-1, to=3-1]
		\arrow["{f_A}"', from=1-2, to=3-2]
		\arrow["f", from=1-3, to=3-3]
		\arrow["{\tau'}"', from=3-1, to=3-2]
		\arrow["{g'}"', from=3-2, to=3-3]
		\arrow["{\Delta'}"', from=3-3, to=3-5]
		\arrow["{f\underset{R}{\otimes}f}", from=1-5, to=2-5]
		\arrow["p", from=2-5, to=3-5]
		\arrow["{\bar{f}}"', curve={height=30pt}, from=1-5, to=3-5]
		\arrow["{\varphi'}", from=3-5, to=3-7]
	\end{tikzcd}\eqno{(\clubsuit)}\]
	
	\noindent  where $\varphi,\varphi'$ and $p$ are the canonical morphisms, $\Delta$ and $\Delta'$ are the diagonal morphisms, and $\bar{f}:=p\circ(f\ten_Rf)$. Notice that for all $b\in B$ one has $$\bar{f}\circ\Delta(b)=\bar{f}(b\ten_R1_B-1_B\ten_Rb)=p(f(b)\ten_R1_{B'}-1_{B'}\ten_Rf(b))$$$$=f(b)\ten_{R'}1_{B'}-1_{B'}\ten_{R'}f(b)=\Delta'(f(b)).$$
	
	Therefore, $(\clubsuit)$ is a commutative diagram. Lipman proved the following proposition.
	
	\begin{prop}\cite[Property 3]{L} \label{prop3}
		Suppose that
		\[\begin{tikzcd}
			R & A & B \\
			{R'} & {A'} & {B'}
			\arrow["\tau", from=1-1, to=1-2]
			\arrow["g", from=1-2, to=1-3]
			\arrow["{\tau'}", from=2-1, to=2-2]
			\arrow["{g'}", from=2-2, to=2-3]
			\arrow["{f_R}"', from=1-1, to=2-1]
			\arrow["{f}", from=1-3, to=2-3]
			\arrow["{f_A}", from=1-2, to=2-2]
		\end{tikzcd}\]
		
		\noindent is a commutative diagram of ring morphisms. Then
		$$f\prtt{A^*_{B,R}} \subseteq \prtt{A'}^*_{B',R'}.$$
	\end{prop}

	Note Proposition \ref{prop3} says that $A^*_{B,R}\sub f^{-1}((A')^*_{B',R'})$, which is a consequence of the inclusion  $$\ker\varphi(B'\ten_{R'}B')\sub\ker\varphi', \footnote{For a ring morphism $h:S\rightarrow T$ and an ideal $I$ of $S$, we use the notation $IT$ to mean the ideal of $T$ generated by $h(I)$.} $$ 
	
	In \cite{SR}, the authors presented sufficient conditions to get the \textit{contraction property} $$f^{-1}((A')^*_{B',R'})=A^*_{B,R}.$$
	
	One of these conditions is $\ker\varphi(B'\ten_{R'}B')$ to be a reduction of $\ker\varphi'$, i.e., $\ker\varphi(B'\ten_{R'}B')\sub\ker\varphi'\sub\overline{\ker\varphi(B'\ten_{R'}B')}$, which is equivalent to $$\overline{\ker\varphi'}=\overline{\ker\varphi(B'\ten_{R'}B')}.$$
	
	This motivated the authors to consider the following diagrams.
	
	\begin{defi}[Lipman diagram]\label{202306211340}
		We say that % https://q.uiver.app/#q=WzAsNixbMCwwLCJSIl0sWzEsMCwiQSJdLFsyLDAsIkIiXSxbMCwxLCJSJyJdLFsxLDEsIkEnIl0sWzIsMSwiQiciXSxbMCwxLCJcXHRhdSJdLFsxLDIsImciXSxbMyw0LCJcXHRhdSciLDJdLFs0LDUsImcnIiwyXSxbMCwzLCJmX1IiLDJdLFsxLDQsImZfQSIsMl0sWzIsNSwiZiIsMl1d
		\[\begin{tikzcd}
			R & A & B \\
			{R'} & {A'} & {B'}
			\arrow["\tau", from=1-1, to=1-2]
			\arrow["g", from=1-2, to=1-3]
			\arrow["{\tau'}", from=2-1, to=2-2]
			\arrow["{g'}", from=2-2, to=2-3]
			\arrow["{f_R}"', from=1-1, to=2-1]
			\arrow["{f}", from=1-3, to=2-3]
			\arrow["{f_A}", from=1-2, to=2-2]
		\end{tikzcd}\] is a \textbf{Lipman diagram} if it is a commutative diagram of ring morphisms and $$\overline{\ker\varphi'}=\overline{\ker\varphi(B'\ten_{R'}B')}.$$
		
		\noindent	We say that it is a \textbf{strong Lipman diagram} if $\ker\varphi'=\ker\varphi(B'\ten_{R'}B').$
	\end{defi} 
	
	To work with the weak normalization, we work with an analogous property, changing the integral closure by the radical. In honor of the fact that Maranesi was the first to address weak normalization in terms of the radical, we will use this to designate the following diagrams.
	
	\begin{defi}[Maranesi diagram]
		We say that % https://q.uiver.app/#q=WzAsNixbMCwwLCJSIl0sWzEsMCwiQSJdLFsyLDAsIkIiXSxbMCwxLCJSJyJdLFsxLDEsIkEnIl0sWzIsMSwiQiciXSxbMCwxLCJcXHRhdSJdLFsxLDIsImciXSxbMyw0LCJcXHRhdSciLDJdLFs0LDUsImcnIiwyXSxbMCwzLCJmX1IiLDJdLFsxLDQsImZfQSIsMl0sWzIsNSwiZiIsMl1d
		\[\begin{tikzcd}
			R & A & B \\
			{R'} & {A'} & {B'}
			\arrow["\tau", from=1-1, to=1-2]
			\arrow["g", from=1-2, to=1-3]
			\arrow["{\tau'}", from=2-1, to=2-2]
			\arrow["{g'}", from=2-2, to=2-3]
			\arrow["{f_R}"', from=1-1, to=2-1]
			\arrow["{f}", from=1-3, to=2-3]
			\arrow["{f_A}", from=1-2, to=2-2]
		\end{tikzcd}\] is a \textbf{Maranesi diagram} if it is a commutative diagram of ring morphisms and $$\sqrt{\ker\varphi'}=\sqrt{\ker\varphi(B'\ten_{R'}B')}.$$
	\end{defi}
	
	Since $\overline{\ker\varphi(B'\ten_{R'}B')}\sub\sqrt{\ker\varphi(B'\ten_{R'}B')}$ then it is clear that every Lipman diagram is a Maranesi diagram. In \cite{SR}, the authors presented a sufficient condition on $f_A$ to get strong Lipman diagrams.
	
	\begin{prop}\cite[Prop. 1.7]{SR}\label{202306011826}
		Consider the commutative diagram % https://q.uiver.app/#q=WzAsNixbMCwwLCJSIl0sWzEsMCwiQSJdLFsyLDAsIkIiXSxbMCwxLCJSJyJdLFsxLDEsIkEnIl0sWzIsMSwiQiciXSxbMCwxLCJcXHRhdSJdLFsxLDIsImciXSxbMyw0LCJcXHRhdSciLDJdLFs0LDUsImcnIiwyXSxbMCwzLCJmX1IiLDJdLFsxLDQsImZfQSIsMl0sWzIsNSwiZiIsMl1d
		\[\begin{tikzcd}
			R & A & B \\
			{R'} & {A'} & {B'}
			\arrow["\tau", from=1-1, to=1-2]
			\arrow["g", from=1-2, to=1-3]
			\arrow["{\tau'}", from=2-1, to=2-2]
			\arrow["{g'}", from=2-2, to=2-3]
			\arrow["{f_R}"', from=1-1, to=2-1]
			\arrow["{f}", from=1-3, to=2-3]
			\arrow["{f_A}", from=1-2, to=2-2]
		\end{tikzcd} \eqno{(\star)}\] \noindent of ring morphisms. If $f_A$ is surjective then $(\star)$ is a strong Lipman diagram. In particular, $(\star)$ is a Maranesi diagram.
	\end{prop}
	
	As a consequence, the sequences studied by Lipman always give rise to strong Lipman diagrams, particularly Maranesi diagrams.
	
	\begin{cor}\cite[Corollary 1.8]{SR}\label{202410241314}
		If  $R\overset{\tau}{\longrightarrow}A\overset{g}{\longrightarrow}B\overset{f}{\longrightarrow}B'$ is a sequence of ring morphism then % https://q.uiver.app/#q=WzAsNixbMCwwLCJSIl0sWzEsMCwiQSJdLFsyLDAsIkIiXSxbMCwxLCJSIl0sWzEsMSwiQSJdLFsyLDEsIkInIl0sWzAsMSwiXFx0YXUiXSxbMSwyLCJnIl0sWzMsNCwiXFx0YXUiLDJdLFs0LDUsImZcXGNpcmMgZyIsMl0sWzAsMywiXFx0ZXh0cm17aWR9X1IiLDJdLFsxLDQsIlxcdGV4dHJte2lkfV9BIiwyXSxbMiw1LCJmIiwyXV0=
		\[\begin{tikzcd}
			R & A & B \\
			R & A & {B'}
			\arrow["\tau", from=1-1, to=1-2]
			\arrow["g", from=1-2, to=1-3]
			\arrow["\tau"', from=2-1, to=2-2]
			\arrow["{f\circ g}"', from=2-2, to=2-3]
			\arrow["{\textrm{id}_R}"', from=1-1, to=2-1]
			\arrow["{\textrm{id}_A}"', from=1-2, to=2-2]
			\arrow["f"', from=1-3, to=2-3]
		\end{tikzcd}\] is a strong Lipman diagram. In particular, it is a Maranesi diagram.
	\end{cor}

	To start our work onward to the contraction property for the weak normalization, we need first to establish the analogous result to the Proposition \ref{prop3}. First, we prove an auxiliary result. 
	
	\begin{lema}\label{202410240016}
		Suppose that % https://q.uiver.app/?q=WzAsNCxbMCwwLCJBIl0sWzEsMCwiQiJdLFswLDEsIkEnIl0sWzEsMSwiQiciXSxbMCwxLCJcXGFscGhhIl0sWzIsMywiXFxhbHBoYSciXSxbMSwzLCJcXHBzaSJdLFswLDIsIlxccGhpIiwyXV0=
		\[\begin{tikzcd}
			A & B \\
			{A'} & {B'}
			\arrow["\alpha", from=1-1, to=1-2]
			\arrow["{\alpha'}", from=2-1, to=2-2]
			\arrow["\psi", from=1-2, to=2-2]
			\arrow["\phi"', from=1-1, to=2-1]
		\end{tikzcd}\]
		\noindent is a commutative diagram of ring morphisms. Then $$\phi(\sqrt{\ker\alpha})\sub\sqrt{\ker\alpha'}.$$ 
		
		If $\phi$ is surjective and $\psi$ is injective then the equality holds.
	\end{lema}
	
	\begin{proof}
		Since $\psi\circ\alpha=\alpha'\circ\phi$ then $\phi(\ker\alpha)\sub \ker\alpha'$. Let $v\in\phi(\sqrt{\ker\alpha})$. Then $v=\phi(u)$, for some $u\in\sqrt{\ker\alpha}$, and consequently, there is $n\in\bN$ such that $u^n\in\ker\alpha$. Thus, $$\alpha'(v^n)=\alpha'(\phi(u^n))=\psi(\alpha(u^n))=\psi(0_B)=0_{B'}$$ $$\implies v^n\in\ker\alpha'\implies v\in\sqrt{\ker\alpha'}.$$
		
		Now, additionally suppose that $\phi$ is surjective and $\psi$ is injective. Conversely, let $v\in\sqrt{\ker\alpha'}$. Then there exists $n\in\bN$ such that $v^n\in\ker\alpha'$. Since $\phi$ is surjective then we can write $v=\phi(u)$, for some $u\in A$. Notice that $$\psi(\alpha(u^n))=\alpha'(\phi(u^n))=\alpha'(v^n)=0_B',$$
		
		\noindent and since $\psi$ is injective then $\alpha(u^n)=0_B$. Hence, $u^n\in\ker\alpha$, i.e., $u\in \sqrt{\ker\alpha}$. Therefore, $v=\phi(u)\in\phi(\sqrt{\ker\alpha})$.
	\end{proof}

	\begin{prop}\label{202410240027}
		Suppose that
		\[\begin{tikzcd}
			R & A & B \\
			{R'} & {A'} & {B'}
			\arrow["\tau", from=1-1, to=1-2]
			\arrow["g", from=1-2, to=1-3]
			\arrow["{\tau'}", from=2-1, to=2-2]
			\arrow["{g'}", from=2-2, to=2-3]
			\arrow["{f_R}"', from=1-1, to=2-1]
			\arrow["{f}", from=1-3, to=2-3]
			\arrow["{f_A}", from=1-2, to=2-2]
		\end{tikzcd}\]
		
		\noindent is a commutative diagram of ring morphisms. Then
		$$f(\widetilde{A}_{B,R}) \subseteq \widetilde{A'}_{B',R'}.$$
	\end{prop}
	
	\begin{proof}
		Since $\widetilde{A}_{B,R}=\Delta^{-1}(\sqrt{\ker\varphi})$ then $\Delta(\widetilde{A}_{B,R})\sub\sqrt{\ker\varphi}$. Thus, $\overline{f}(\Delta(\widetilde{A}_{B,R}))\sub\overline{f}(\sqrt{\ker\varphi})$, and once $(\clubsuit)$ is a commutative diagram, Lemma \ref{202410240016} implies that $\Delta'(f(\widetilde{A}_{B,R}))\sub\sqrt{\ker\varphi'}$. Hence, $$f(\widetilde{A}_{B,R})\sub \Delta'^{ -1}(\sqrt{\ker\varphi'})=\widetilde{A'}_{B',R'}.$$
	\end{proof}
	
	\begin{cor}
		
		If $R\overset{\tau}{\longrightarrow}A\overset{\lambda}{\longrightarrow}C\overset{g_C}{\longrightarrow} B$ is a sequence of ring morphisms and $g=g_C\circ\lambda$ then $$A^*_{B,R}\sub C^*_{B,R}.$$
		
		In particular, if $A\sub C$ and $\lambda$ is the inclusion map, one has $A^*_{B,R}\sub C^*_{B,R}$.
	\end{cor}
	
	\begin{proof}
		It is a consequence of Proposition \ref{202410240027} applied to the commutative diagram % https://q.uiver.app/#q=WzAsNixbMCwwLCJSIl0sWzEsMCwiQSJdLFsyLDAsIkIiXSxbMCwxLCJSIl0sWzEsMSwiQyJdLFsyLDEsIkIiXSxbMCwxLCJcXHRhdSJdLFsxLDIsImciXSxbMCwzLCJcXHRleHRybXtpZH1fUiIsMl0sWzEsNCwiXFxsYW1iZGEiLDJdLFszLDQsIlxcbGFtYmRhXFxjaXJjXFx0YXUiLDJdLFs0LDUsImdfQyIsMl0sWzIsNSwiXFx0ZXh0cm17aWR9X0IiXV0=
		\[\begin{tikzcd}
			R & A & B \\
			R & C & B
			\arrow["\tau", from=1-1, to=1-2]
			\arrow["{\textrm{id}_R}"', from=1-1, to=2-1]
			\arrow["g", from=1-2, to=1-3]
			\arrow["\lambda"', from=1-2, to=2-2]
			\arrow["{\textrm{id}_B}", from=1-3, to=2-3]
			\arrow["{\lambda\circ\tau}"', from=2-1, to=2-2]
			\arrow["{g_C}"', from=2-2, to=2-3]
		\end{tikzcd}.\]
	\end{proof}
	
	We continue by recalling some useful observations for this work.
	
	\begin{lema}\label{202410240051}
		If $h: S\rightarrow T$ is a ring morphism.
		
		\begin{enumerate}
			\item [a)] If $I$ is an ideal of $S$ then $h(\sqrt{I})\sub\sqrt{IT}$;
			
			\item [b)] If $J$ is an ideal of $T$ then $h^{-1}(\sqrt{J})=\sqrt{h^{-1}(J)}$.
		\end{enumerate}
	\end{lema}
	
	\begin{proof}
		(a) If $v\in h(\sqrt{I})$ then $v=h(u)$ for some $u\in\sqrt{I}$, and consequently, there exists $n\in\bN$ such that $u^n\in I$. Hence, $v^n=h(u^n)\in h(I)\sub IT$, and therefore, $v\in\sqrt{IT}$.
		
		(b) Notice that: $x\in h^{-1}(\sqrt{J})\iff h(x)\in \sqrt{J}\iff \exists n\in\bN, h(x)^n\in J$
		
		$$\iff \exists n\in\bN, h(x^n)\in J\iff \exists n\in\bN,  x^n\in h^{-1}(J)\iff x\in\sqrt{h^{-1}(J)}.$$
	\end{proof}
	
	\begin{lema}\label{202411031636}
		If $\alpha:R\rightarrow S$ and $\beta:S\rightarrow T$ are ring morphism, $\ker\alpha$ and $\ker\beta$ are nil ideals then $\ker(\beta\circ\alpha)$ is a nil ideal.
	\end{lema}
	
	\begin{proof}
		Let $v\in\ker(\beta\circ\alpha)$. Thus, $\alpha(v)\in\ker\beta\sub\sqrt{\ideal{0_S}}$, and then:  $$ v\in\alpha^{-1}(\sqrt{\ideal{0_S}})=\sqrt{\alpha^{-1}(\ideal{0_S})}=\sqrt{\ker\alpha}\sub \sqrt{\sqrt{\ideal{0_R}}}=\sqrt{\ideal{0_R}}.$$
		
		Therefore, $v\in\sqrt{\ideal{0_R}}$.
	\end{proof}

	We end this section recalling the notion of radicial algebras/morphisms.\footnote{Do not confuse with \textit{radical}.}
	
	\begin{defi}
		Let $h:S\rightarrow T$ be a ring morphism. We say that $T$ is a \textbf{radicial} $S$-algebra if:
		
		\begin{enumerate}
			\item The induced map on spectra $\Spec(h):\Spec T\rightarrow \Spec S$ is injective;
			
			\item For every $\fq\in\spec S$ the embedding $h_{\fq}:\Frac\left(\dfrac{S}{h^{-1}(\fq)}\right)\hookrightarrow \Frac\left(\dfrac{T}{\fq}\right)$ induced by $h$ is purely inseparable.\footnote{$\Frac$ stands for the field of fractions.}
		\end{enumerate}
		
		In this case, we also say that $h$ is a radicial ring morphism. 
	\end{defi}
	
	The following theorem describes equivalent conditions for a radicial ring morphism which will be useful to this work.

	\begin{teo}\cite[Proposition 3.7.1, p. 246]{EGA}\label{202411171813} The following conditions are equivalent:
		
		\begin{enumerate}
			\item [(a)] $h:S\rightarrow T$ is a radicial ring morphism;
			
			\item [(b)] Any two distinct morphisms of $T$ into a field have distinct composition with $h$;
			
			\item [(c)]  The kernel of the canonical morphism $\gamma: T\ten_ST\rightarrow T$ is a nil ideal of $T\ten_ST$;
			
			\item [(d)] $t\ten_S 1-1\ten_St$ is nilpotent in $T\ten_ST$, $\forall t\in T$.
		\end{enumerate} 
		
	\end{teo}
	
	\begin{example}\label{202411172308}
		Let $C$ be a ring and $S$ be a multiplicative subset of $C$. We will prove that the localization map $$\begin{matrix}
			\rho: & C & \longrightarrow & S^{-1}C\\
			& x & \longmapsto & \dfrac{x}{1}
		\end{matrix}$$\noindent is a radical ring morphism. To prove it, we will show that $\rho$ satisfies condition (b) of Theorem \ref{202411171813}. Let $K$ be a field and $g_1,g_2:S^{-1}C\rightarrow K$ be two ring morphisms, and suppose that $g_1\circ\rho=g_2\circ\rho$. We have to show that $g_1=g_2$. If $u\in S^{-1}C$ then we can write $u=\frac{x}{s}$, with $x\in C$ and $s\in S$. Then: $$g_1(u)=g_1\left(\frac{x}{1}\left(\frac{s}{1}\right)^{-1}\right)=g_1\left(\frac{x}{1}\right)\left(g_1\left(\frac{s}{1}\right)\right)^{-1}=g_2\left(\frac{x}{1}\right)\left(g_2\left(\frac{s}{1}\right)\right)^{-1}$$ $$=g_2\left(\frac{x}{1}\left(\frac{s}{1}\right)^{-1}\right)=g_2(u).$$
		
		Therefore, $\rho$ is a radicial ring morphism.
	\end{example}

	\section{Universally injective contraction}
	
	In \cite{SR}, the authors proved that if 	\[\begin{tikzcd}
		R & A & B \\
		{R'} & {A'} & {B'}
		\arrow["\tau", from=1-1, to=1-2]
		\arrow["g", from=1-2, to=1-3]
		\arrow["{\tau'}", from=2-1, to=2-2]
		\arrow["{g'}", from=2-2, to=2-3]
		\arrow["{f_R}"', from=1-1, to=2-1]
		\arrow["{f}", from=1-3, to=2-3]
		\arrow["{f_A}", from=1-2, to=2-2]
	\end{tikzcd}\]
	
	\noindent is a Lipman diagram then $$f^{-1}((A')^*_{B',R'})=A^*_{B,R},$$
	\noindent under the condition of universal injectivity. The main goal of this section is to prove that Maranesi diagrams have the analogous property for weak normalization. 
	
	Recall that a ring morphism $\alpha:S\rightarrow S'$ is universally injective if and only if for every $S$-module $N$ the induced morphism

	$$\begin{matrix}
		\tilde{\alpha}: & N & \longrightarrow  &  N\ten_SS'\\
		& n & \longmapsto         & n\ten_S 1_{S'}
	\end{matrix}$$

	\noindent is injective.
	The next proposition will be necessary to obtain the main result of this section.
	
	\begin{prop}\label{202410241303}
		If $h:S\rightarrow T$ is an universally injective ring morphism and $I$ is an ideal of $S$ then $$h^{-1}(\sqrt{IT})=\sqrt{I}.$$
	\end{prop}
	
	\begin{proof}
		Since $h$ is universally injective then the proof of \cite[Lemma 2.3]{SR} shows that $h^{-1}(IT)=I$. By Lemma \ref{202410240051} (b) we conclude that $$h^{-1}(\sqrt{IT})=\sqrt{h^{-1}(IT)}=\sqrt{I}.$$
	\end{proof}
	
	The next results were proved in \cite{SR} and it will be an important tool for the proof of the main theorem of this section.
	
	\begin{prop}\cite[Corollary 2.2]{SR}\label{202410241259}
		If $\alpha:S\rightarrow S'$ and $\beta:T\rightarrow T'$ are universally injective morphisms of $R$-algebras then $$\alpha\ten_R\beta:S\ten_RT\rightarrow S'\ten_RT'$$
		
		\noindent is universally injective.
	\end{prop}
	
	The next result was proved in \cite{SR}, and it will be an important tool to establish some theorems in this work. 
	
	\begin{lema}\cite[Lemma 3.2]{SR}\label{202410241420}
		Let $h: S\rightarrow T$ be an integral morphism of $R$-algebras. 
		
		\begin{enumerate}
			\item [a)] $\ker(h\ten_R h)$ is a nil ideal of $S\ten_RS$;
			
			\item [b)] Suppose that $\ker h$ is a nil ideal of $S$. Then $\overline{I}=h^{-1}(\overline{IT})$, for every $I$ ideal of $S$.
		\end{enumerate}
	\end{lema}
	
	Now we adapt item (b) of the previous result for the radical of ideals.
	
	\begin{lema}\label{202410251426}
		Let $h: S\rightarrow T$ be an integral ring morphism. If $\ker h$ is a nilideal of $S$ then $h^{-1}(\sqrt{IT})=\sqrt{I}$, for every ideal $I$ of $S$.
	\end{lema}
	
	\begin{proof}
		By Lemma \ref{202410240051} (a) we have $\sqrt{I}\sub h^{-1}(\sqrt{IT})$. Also, Lemma \ref{202410241420} implies that $h^{-1}(\overline{IT})=\overline{I}$. Using Lemma \ref{202410240051} (b), we obtain $$h^{-1}(\sqrt{IT})=\sqrt{h^{-1}(IT)}\sub \sqrt{h^{-1}(\overline{IT})}=\sqrt{\overline{I}}\sub\sqrt{\sqrt{I}}=\sqrt{I}.$$
	\end{proof}
	
	\begin{prop}\label{202411010055}
		Suppose that
		\[\begin{tikzcd}
			R & A & B \\
			{R'} & {A'} & {B'}
			\arrow["\tau", from=1-1, to=1-2]
			\arrow["g", from=1-2, to=1-3]
			\arrow["{\tau'}", from=2-1, to=2-2]
			\arrow["{g'}", from=2-2, to=2-3]
			\arrow["{f_R}"', from=1-1, to=2-1]
			\arrow["{f}", from=1-3, to=2-3]
			\arrow["{f_A}", from=1-2, to=2-2]
		\end{tikzcd}\]
		
		\noindent is a Maranesi diagram. If $\ker p$ is a nil ideal and $f$ is universally injective then $$f^{-1}(\widetilde{A'}_{B',R'})=\widetilde{A}_{B,R}.$$
	\end{prop}
	
	\begin{proof}
		Since $f$ is universally injective then Proposition \ref{202410241259} ensures that $f\ten_R f$ is universally injective. Thus, by Proposition \ref{202410241303} we have $$(f\ten_Rf)^{-1}(\sqrt{\ker\varphi(B'\ten_RB')})=\sqrt{\ker\varphi}. \eqno{(1)}$$
		
		Consider the ideal $I:=\ker\varphi(B'\ten_RB')$. Thus, $I(B'\ten_{R'}B')=\ker\varphi(B'\ten_{R'}B')$. Since $p$ is surjective (hence, integral) and $\ker p$ is a nil ideal then Lemma \ref{202410251426} implies that $$p^{-1}(\sqrt{\ker\varphi(B'\ten_{R'}B')})=p^{-1}(\sqrt{I(B'\ten_{R'}B')})=\sqrt{I}=\sqrt{\ker\varphi(B'\ten_RB')}.\eqno(2)$$
		
		Since $\bar{f}=p\circ(f\ten_Rf)$ then, using the fact that we are working with a Maranesi diagram, equations (1) and (2) imply that $\overline{f}^{-1}(\sqrt{\ker\varphi'})=\overline{f}^{-1}(\sqrt{\ker\varphi(B'\ten_{R'}B')})=\sqrt{\ker\varphi}$. We already have seen that % https://q.uiver.app/#q=WzAsNCxbMiwwLCJCXFx1bmRlcnNldHtSfXtcXG90aW1lc31CIl0sWzIsMSwiQidcXHVuZGVyc2V0e1InfXtcXG90aW1lc31CJyJdLFswLDAsIkIiXSxbMCwxLCJCJyJdLFsyLDAsIlxcRGVsdGEiXSxbMywxLCJcXERlbHRhJyJdLFsyLDMsImYiLDJdLFswLDEsIlxcYmFye2Z9Il1d
		\[\begin{tikzcd}
			B && {B\underset{R}{\otimes}B} \\
			{B'} && {B'\underset{R'}{\otimes}B'}
			\arrow["\Delta", from=1-1, to=1-3]
			\arrow["{\Delta'}", from=2-1, to=2-3]
			\arrow["f"', from=1-1, to=2-1]
			\arrow["{\bar{f}}", from=1-3, to=2-3]
		\end{tikzcd}\] \noindent commutes, and consequently
		
		$$f^{-1}(\widetilde{A'}_{B', R'})=f^{-1}(\Delta'^{-1}(\sqrt{\ker\varphi'}))=\Delta^{-1}(\bar{f}^{-1}(\sqrt{\ker\varphi'}))=\Delta^{-1}(\sqrt{\ker\varphi})=\widetilde{A}_{B,R}.$$
	\end{proof}
	
	\begin{cor}
		Consider a sequence of ring morphisms $R\longrightarrow R'\longrightarrow A\overset{g}{\longrightarrow} B$. If    $\ker(B\ten_RB\rightarrow B\ten_{R'}B)$ (which is $\ker p$) is a nil ideal then $$\widetilde{A}_{B,R}=\widetilde{A}_{B,R'}.$$ 
	\end{cor}
	
	\begin{proof}
		From the above sequence, we get a strong Lipman diagram (hence a Maranesi diagram) % https://q.uiver.app/#q=WzAsNixbMCwwLCJSIl0sWzIsMCwiQiJdLFswLDEsIlInIl0sWzEsMSwiQSJdLFsyLDEsIkIiXSxbMSwwLCJBIl0sWzAsNV0sWzAsMiwiZl9SIiwyXSxbMiwzXSxbMyw0LCJnIiwyXSxbNSwxLCJnIl0sWzUsMywiXFx0ZXh0cm17aWR9X0EiXSxbMSw0LCJcXHRleHRybXtpZH1fQiJdXQ==
		\[\begin{tikzcd}
			R & A & B \\
			{R'} & A & B
			\arrow[from=1-1, to=1-2]
			\arrow["{f_R}"', from=1-1, to=2-1]
			\arrow[from=2-1, to=2-2]
			\arrow["g"', from=2-2, to=2-3]
			\arrow["g", from=1-2, to=1-3]
			\arrow["{\textrm{id}_A}", from=1-2, to=2-2]
			\arrow["{\textrm{id}_B}", from=1-3, to=2-3]
		\end{tikzcd}\]\noindent  and since $\id_B$ is integral and $\ker p$ is a nil ideal then Proposition \ref{202410241734} implies  $\id_B^{-1}(\widetilde{A}_{B,R'})=\widetilde{A}_{B,R}$. Hence, $\widetilde{A}_{B,R}=\widetilde{A}_{B,R'}$.
	\end{proof}
	
	\begin{lema}\label{202410312042}
		If $f_R:R\rightarrow R'$ is a radicial ring morphism and $C$ is an $R'$-algebra then the kernel of the canonical morphism $p_C:C\ten_RC\rightarrow C\ten_{R'}C$ is a nil ideal of $C\ten_RC$.
	\end{lema}
	
	\begin{proof}
		Since $f_R$ is radicial then $\ker\gamma\sub\sqrt{\ideal{0_{R'\ten_RR'}}}$, where $\gamma: R'\ten_RR'\rightarrow R'$ is the canonical morphism. Let $\lambda:R'\rightarrow C$ be the structure morphism of $C$. We know that $$\ker p_C=(\Delta_C\circ \lambda)(R')(C\ten_RC),$$\noindent  where $\Delta_C:C\rightarrow C\ten_RC$ is the diagonal map.  It is easy to see that, if $\sigma:R'\ten_RR'\rightarrow B'\ten_RB'$ is given by $\sigma:=\lambda\ten_R\lambda$, then $$\Delta_C\circ \lambda(r')=\sigma(r'\ten_R1_{R'}-1_{R'}\ten_Rr'), \forall r'\in R'.$$
		
		Since $\ker\gamma=\ideal{\{r'\ten_R1_{R'}-1_{R'}\ten_Rr'\mid r'\in R'\}}$ then $$\ker p_C=(\sigma(\ker\gamma))(C\ten_RC).$$ Since $\ker\gamma\sub\sqrt{\ideal{0_{R'\ten_RR'}}}$ then $\sigma(\ker\gamma)\sub\sigma(\sqrt{\ideal{0_{R'\ten_RR'}}})\sub \sqrt{\ideal{0_{C\ten_RC}}}$, which implies that $\ker p_C$ is a nil ideal of $C\ten_RC$.
	\end{proof}
	
	Now we state the main result of this section.
	
	\begin{teo}[\textbf{Radicial + Universally injective} contraction]\label{202410241447} 	Suppose that
		\[\begin{tikzcd}
			R & A & B \\
			{R'} & {A'} & {B'}
			\arrow["\tau", from=1-1, to=1-2]
			\arrow["g", from=1-2, to=1-3]
			\arrow["{\tau'}", from=2-1, to=2-2]
			\arrow["{g'}", from=2-2, to=2-3]
			\arrow["{f_R}"', from=1-1, to=2-1]
			\arrow["{f}", from=1-3, to=2-3]
			\arrow["{f_A}", from=1-2, to=2-2]
		\end{tikzcd}\]
		
		\noindent is a Maranesi diagram. If $f_R$ is radicial and $f$ is universally injective then $$f^{-1}(\widetilde{A'}_{B',R'})=\widetilde{A}_{B,R}.$$
	\end{teo}
	
	\begin{proof}
		Since $f_R$ is radicial then Lemma \ref{202410312042} ensures that $\ker p$ is a nil ideal. Now, the result is a direct consequence of Proposition \ref{202411010055}. 
	\end{proof}
	
	As a consequence, we obtain for weak normalization a result similar to what Lipman achieved for the sequences he studied in the case of relative Lipschitz saturation (under the assumption of being faithfully flat \cite{L}), but with the hypothesis of universal injectivity.
	
	\begin{cor}\label{202305241951}
		Let $R\overset{\tau}{\rightarrow} A \overset{g}{\rightarrow}B\overset{f}{\rightarrow}B'$ be a sequence of ring morphisms, and assume that $f$ is universally injective. Then: $$\widetilde{A}_{B,R}=f^{-1}(\widetilde{A}_{B',R}).$$
	\end{cor}
	
	\begin{proof}
		In Corollary \ref{202410241314} we have seen that the above sequence gives rise to a Maranesi diagram % https://q.uiver.app/#q=WzAsNixbMCwwLCJSIl0sWzEsMCwiQSJdLFsyLDAsIkIiXSxbMCwxLCJSIl0sWzEsMSwiQSJdLFsyLDEsIkInIl0sWzAsMSwiXFx0YXUiXSxbMSwyLCJnIl0sWzMsNCwiXFx0YXUiLDJdLFs0LDUsImZcXGNpcmMgZyIsMl0sWzAsMywiXFx0ZXh0cm17aWR9X1IiLDJdLFsxLDQsIlxcdGV4dHJte2lkfV9BIiwyXSxbMiw1LCJmIiwyXV0=
		\[\begin{tikzcd}
			R & A & B \\
			R & A & {B'}
			\arrow["\tau", from=1-1, to=1-2]
			\arrow["g", from=1-2, to=1-3]
			\arrow["\tau"', from=2-1, to=2-2]
			\arrow["{f\circ g}"', from=2-2, to=2-3]
			\arrow["{\textrm{id}_R}"', from=1-1, to=2-1]
			\arrow["{\textrm{id}_A}"', from=1-2, to=2-2]
			\arrow["f"', from=1-3, to=2-3]
		\end{tikzcd}\].
		
		In this case, $f_R=\id_R$ is radicial and $f$ is universally injective, and the proof is done.
	\end{proof}
	
	Another consequence of Theorem \ref{202410241447} is the possibility of changing the ground ring when computing the weak normalization, as long as the base rings are connected by a radicial morphism.

	%Lemma \ref{202410312042}
	%Corollary \ref{202410251558}
	
	\begin{cor}\label{202410251558}
		Consider a sequence of ring morphisms $R\overset{f_R}{\longrightarrow} R' \overset{}{\longrightarrow}A\overset{g}{\longrightarrow}B$. If $f_R$ is radicial then $$\widetilde{A}_{B,R}=\widetilde{A}_{B,R'}.$$ 
	\end{cor}
	
	\begin{proof}
		From the above sequence, we get a strong Lipman diagram (hence, a Maranesi diagram) % https://q.uiver.app/#q=WzAsNixbMCwwLCJSIl0sWzIsMCwiQiJdLFswLDEsIlInIl0sWzEsMSwiQSJdLFsyLDEsIkIiXSxbMSwwLCJBIl0sWzAsNV0sWzAsMiwiZl9SIiwyXSxbMiwzXSxbMyw0LCJnIiwyXSxbNSwxLCJnIl0sWzUsMywiXFx0ZXh0cm17aWR9X0EiXSxbMSw0LCJcXHRleHRybXtpZH1fQiJdXQ==
		\[\begin{tikzcd}
			R & A & B \\
			{R'} & A & B
			\arrow[from=1-1, to=1-2]
			\arrow["{f_R}"', from=1-1, to=2-1]
			\arrow[from=2-1, to=2-2]
			\arrow["g"', from=2-2, to=2-3]
			\arrow["g", from=1-2, to=1-3]
			\arrow["{\textrm{id}_A}", from=1-2, to=2-2]
			\arrow["{\textrm{id}_B}", from=1-3, to=2-3]
		\end{tikzcd}\]\noindent  and since $\id_B$ is universally injective then Theorem \ref{202410241447} yields $\id_B^{-1}(\widetilde{A}_{B,R'})=\widetilde{A}_{B,R}$. Hence, $\widetilde{A}_{B,R}=\widetilde{A}_{B,R'}$.
	\end{proof}
	
	Next, we explore other consequences of Theorem \ref{202410241447}. First, we will prove that some conditions on $f_A$ and $f$ ensure that the standard commutative diagram is Maranesi.
	
	\begin{prop}\label{202411031659}
		Suppose that
		\[\begin{tikzcd}
			R & A & B \\
			{R'} & {A'} & {B'}
			\arrow["\tau", from=1-1, to=1-2]
			\arrow["g", from=1-2, to=1-3]
			\arrow["{\tau'}", from=2-1, to=2-2]
			\arrow["{g'}", from=2-2, to=2-3]
			\arrow["{f_R}"', from=1-1, to=2-1]
			\arrow["{f}", from=1-3, to=2-3]
			\arrow["{f_A}", from=1-2, to=2-2]
		\end{tikzcd}\eqno{(\star)}\]
		
		\noindent is a commutative diagram of ring morphisms. If $f_A$ is radicial and $f$ is surjective then $(\star)$ is a Maranesi diagram.
	\end{prop}
	
	\begin{proof}
		Since $f_A$ is radicial, by Lemma \ref{202410312042} one has $\ker p_A\sub\sqrt{\ideal{0_{B'\ten_AB'}}}$. As $f$ is surjective, consequently $f$ is integral, and Lemma \ref{202410241420} (a) ensures that $\ker(f\ten_Af)\sub\sqrt{\ideal{0_{B\ten_AB}}}$. By Lemma \ref{202411031636} we conclude that $\ker \overline{f_A}\sub\sqrt{\ideal{0_{B\ten_AB}}}$, where $\overline{f_A}:=p_A\circ(f\ten_Af)$.

		Once $f$ is surjective, clearly $f\ten_Rf$ is surjective, and the surjectivity of $p$ implies that $\overline{f}=p\circ(f\ten_Rf)$ is surjective. Hence, $$\ker\varphi(B'\ten_{R'}B')=\overline{f}(\ker\varphi).$$
		
		Thus, to prove that $(\star)$ is a Maranesi diagram, it suffices to prove that $\ker\varphi'\sub\sqrt{\overline{f}(\ker\varphi)}$. Let $v\in\ker\varphi'$. Since $\overline{f}$ is surjective then $v=\overline{f}(u)$, for some $u\in B\ten_RB$. Thus, $$\overline{f_A}(\varphi(u))=\varphi'(\overline{f}(u))=\varphi'(v)=0_{B'\ten_{R'}B'},$$
		
		\noindent and this implies $\varphi(u)\in\ker\overline{f_A}\sub\sqrt{\ideal{0_{B\ten_AB}}}$. Hence, there exists $r\in\bN$ such that $(\varphi(u))^r=0_{B\ten_AB}$, i.e., $\varphi(u^r)=0_{B\ten_AB}$. Therefore, $$u^r\in \ker\varphi\implies v^r=\overline{f}(u^r)\in\overline{f}(\ker\varphi)\implies v\in\sqrt{\overline{f}(\ker\varphi)}.$$
	\end{proof}

	%radicial Lemma \ref{202410312042} 
	
	We emphasize that in the next result, we do not need to require that the diagram be a Maranesi diagram, as the listed assumptions will automatically ensure this.
	
	\begin{cor}
		Suppose that % https://q.uiver.app/#q=WzAsNSxbMCwwLCJSIl0sWzEsMCwiQSJdLFsyLDEsIkIiXSxbMCwyLCJSJyJdLFsxLDIsIkEnIl0sWzAsMywiZl9SIl0sWzEsNCwiZl9BIl0sWzEsMiwiZyJdLFs0LDIsImcnIiwyXSxbMyw0LCJcXHRhdSciXSxbMCwxLCJcXHRhdSJdXQ==
		\[\begin{tikzcd}
			R & A \\
			&& B \\
			{R'} & {A'}
			\arrow["{f_R}", from=1-1, to=3-1]
			\arrow["{f_A}", from=1-2, to=3-2]
			\arrow["g", from=1-2, to=2-3]
			\arrow["{g'}"', from=3-2, to=2-3]
			\arrow["{\tau'}", from=3-1, to=3-2]
			\arrow["\tau", from=1-1, to=1-2]
		\end{tikzcd} \eqno{(\star)}\] \noindent is a commutative diagram of ring morphisms. If $f_R$ and $f_A$ are radicial then $$\widetilde{A}_{B,R}=\widetilde{A'}_{B,R'}.$$
	\end{cor}
	
	\begin{proof}
		Since $f_A$ is radicial and $\id_B$ is surjective then Proposition \ref{202411031659} guarantees that $(\star)$ is a Maranesi diagram. Since $f_R$ is radicial and $\id_B$ is universally injective then Theorem \ref{202410241447} implies that $\id_B^{-1}(\widetilde{A'}_{B,R'})=\widetilde{A}_{B,R}$, i.e., $$\widetilde{A'}_{B,R'}=\widetilde{A}_{B,R}.$$
	\end{proof}

	As a direct consequence, we obtain a version of the previous proposition formulated for the sequences analyzed by Lipman in \cite{L}.
	
	\begin{cor}
		If  $R\overset{}{\longrightarrow} A \overset{f_A}{\longrightarrow}A'\overset{}{\longrightarrow}B$ is a sequence of ring morphisms and $f_A$ is radicial then $$\widetilde{A}_{B,R}=\widetilde{A'}_{B,R}.$$
	\end{cor}

	In the following, we employ Theorem \ref{202410241447} to show that idempotency holds for the weak normalization even if $A$ is not necessarily an $R$-subalgebra of $B$.
	
	\begin{prop}
		If $R\overset{\tau}{\longrightarrow}A\overset{g}{\longrightarrow}B$ is a sequence of ring morphisms then $$\widetilde{(\widetilde{A}_{B,R})}_{B,R}=\widetilde{A}_{B,R}.$$ 
	\end{prop}
	
	\begin{proof}
		Consider the ring morphism $f_A:A\rightarrow \widetilde{A}_{B,R}$ given by $f_A(a):=g(a)$, $\forall  a\in A$ and let $g':\widetilde{A}_{B,R}\hookrightarrow B$ be the inclusion map. Thus it is clear that % https://q.uiver.app/#q=WzAsNixbMCwwLCJSIl0sWzAsMSwiUiJdLFsxLDAsIkEiXSxbMSwxLCJBXipfe0IsUn0iXSxbMiwwLCJCIl0sWzIsMSwiQiJdLFswLDIsIlxcdGF1Il0sWzIsNCwiZyJdLFsxLDMsImZfQVxcY2lyY1xcdGF1IiwyXSxbMiwzLCJmX0EiLDJdLFswLDEsIlxcdGV4dHJte2lkfV9SIiwyXSxbNCw1LCJcXHRleHRybXtpZH1fQiJdLFszLDUsImcnIiwyLHsic3R5bGUiOnsidGFpbCI6eyJuYW1lIjoiaG9vayIsInNpZGUiOiJ0b3AifX19XV0=
		\[\begin{tikzcd}
			R & A & B \\
			R & {\widetilde{A}_{B,R}} & B
			\arrow["\tau", from=1-1, to=1-2]
			\arrow["g", from=1-2, to=1-3]
			\arrow["{f_A\circ\tau}"', from=2-1, to=2-2]
			\arrow["{f_A}"', from=1-2, to=2-2]
			\arrow["{\textrm{id}_R}"', from=1-1, to=2-1]
			\arrow["{\textrm{id}_B}", from=1-3, to=2-3]
			\arrow["{g'}"', hook, from=2-2, to=2-3]
		\end{tikzcd}\eqno{(\star)}\]
		
		\noindent commutes. We know that the kernel of the morphism $\varphi':B\ten_RB\rightarrow B\ten_{\widetilde{A}_{B,R}}B$ is generated by $\Delta(g'(\widetilde{A}_{B,R}))=\Delta(\widetilde{A}_{B,R})$. Since $\widetilde{A}_{B,R}=\Delta^{-1}(\sqrt{\ker\varphi})$ then $\ker\varphi'\sub\sqrt{\ker\varphi}$. As a result, $(\star)$ becomes a Maranesi diagram. Moreover, since $\id_R$ is radicial and $\id_B$ is universally injective then Theorem \ref{202410241447} ensures that $\id_B^{-1}(\widetilde{(\widetilde{A}_{B,R})}_{B,R})=\widetilde{A}_{B,R}$, implying $\widetilde{(\widetilde{A}_{B,R})}_{B,R}=\widetilde{A}_{B,R}$.
	\end{proof}

	\section{Integral contraction}
	
	In \cite{SR}, the authors proved the following theorem.
	
	\begin{teo}\cite[Theorem 3.3]{SR}\label{202410262320}	Suppose that
		\[\begin{tikzcd}
			R & A & B \\
			{R'} & {A'} & {B'}
			\arrow["\tau", from=1-1, to=1-2]
			\arrow["g", from=1-2, to=1-3]
			\arrow["{\tau'}", from=2-1, to=2-2]
			\arrow["{g'}", from=2-2, to=2-3]
			\arrow["{f_R}"', from=1-1, to=2-1]
			\arrow["{f}", from=1-3, to=2-3]
			\arrow["{f_A}", from=1-2, to=2-2]
		\end{tikzcd}\]
		
		\noindent is a Lipman diagram. If $f$ is an integral morphism and the kernel of the canonical morphism $p: B'\ten_RB'\rightarrow B'\ten_{R'}B'$ is a nil ideal then $$f^{-1}((A')^*_{B',R'})=A^*_{B,R}.$$
	\end{teo}

	In this section, we will see that Maranesi diagrams are suitable for obtaining the contraction property when working with the integrality hypothesis.
	
	\begin{teo}[\textbf{Integral} contraction]\label{202410241734}
		Let \[\begin{tikzcd}
			R & A & B \\
			{R'} & {A'} & {B'}
			\arrow["\tau", from=1-1, to=1-2]
			\arrow["g", from=1-2, to=1-3]
			\arrow["{\tau'}", from=2-1, to=2-2]
			\arrow["{g'}", from=2-2, to=2-3]
			\arrow["{f_R}"', from=1-1, to=2-1]
			\arrow["{f}", from=1-3, to=2-3]
			\arrow["{f_A}", from=1-2, to=2-2]
		\end{tikzcd}\] be a Maranesi diagram. If $f$ is an integral morphism and the kernel of the canonical morphism $p: B'\ten_RB'\rightarrow B'\ten_{R'}B'$ is a nil ideal then  $$f^{-1}(\widetilde{A'}_{B',R'})=\widetilde{A}_{B,R}.$$
	\end{teo}
	
	\begin{proof}
		Since $f$ is integral then $f\ten_Rf$ is integral. By Lemma \ref{202410241420} (a), $\ker(f\ten_Rf)$ is a nil ideal, i.e., $\ker(f\ten_Rf)\sub\sqrt{\ideal{0_{B\ten_RB}}}$. Further, $p$ is surjective, so it is integral, which implies that $\bar{f}=p\circ(f\ten_Rf)$ is integral. Besides, using that $\ker p$ is a nil ideal, we have $$\ker\bar{f}=(f\ten_Rf)^{-1}(\ker p)\sub(f\ten_Rf)^{-1}(\sqrt{\ideal{0_{B\ten_RB}}})\sub\sqrt{(f\ten_Rf)^{-1}(\ideal{0_{B\ten_RB}})}$$$$=\sqrt{\ker(f\ten_Rf)}\sub\sqrt{\sqrt{\ideal{0_{B\ten_RB}}}}=\sqrt{\ideal{0_{B\ten_RB}}}.$$
		
		Thus, $\ker\bar{f}$ is a nil ideal, and Lemma \ref{202410251426} implies that $\bar{f}^{-1}(\sqrt{\ker\varphi(B'\ten_{R'}B')})=\sqrt{\ker\varphi}$. From now on, it suffices to proceed exactly as we did in the proof of Proposition  \ref{202411010055}.
	\end{proof}
	
	Theorem \ref{202410241734} also can be applied to obtain the contraction property for the weak normalization in the sequences studied by Lipman in \cite{L}.
	
	\begin{cor}\label{202410251524}
		Let $R\overset{\tau}{\rightarrow} A \overset{g}{\rightarrow}B\overset{f}{\rightarrow}B'$ be a sequence of ring morphisms, and assume that $f$ is integral.  Then: $$f^{-1}(\widetilde{A}_{B',R})=\widetilde{A}_{B,R}.$$
		
	\end{cor}
	
	\begin{proof}
		In Corollary \ref{202410241314} we have seen the sequence gives rise to a Maranesi diagram % https://q.uiver.app/#q=WzAsNixbMCwwLCJSIl0sWzEsMCwiQSJdLFsyLDAsIkIiXSxbMCwxLCJSIl0sWzEsMSwiQSJdLFsyLDEsIkInIl0sWzAsMSwiXFx0YXUiXSxbMSwyLCJnIl0sWzMsNCwiXFx0YXUiLDJdLFs0LDUsImZcXGNpcmMgZyIsMl0sWzAsMywiXFx0ZXh0cm17aWR9X1IiLDJdLFsxLDQsIlxcdGV4dHJte2lkfV9BIiwyXSxbMiw1LCJmIiwyXV0=
		\[\begin{tikzcd}
			R & A & B \\
			R & A & {B'}
			\arrow["\tau", from=1-1, to=1-2]
			\arrow["g", from=1-2, to=1-3]
			\arrow["\tau"', from=2-1, to=2-2]
			\arrow["{f\circ g}"', from=2-2, to=2-3]
			\arrow["{\textrm{id}_R}"', from=1-1, to=2-1]
			\arrow["{\textrm{id}_A}"', from=1-2, to=2-2]
			\arrow["f"', from=1-3, to=2-3]
		\end{tikzcd}\].
		
		In this case, $p$ is the identity morphism, and in particular, its kernel is a nil ideal, which finishes the proof once $f$ is integral.
	\end{proof}
	
	We apply the integral contraction to show that the weak normalization commutes with the quotient by an ideal.
	
	\begin{prop}\label{202410251529}
		Let $R\overset{\tau}{\rightarrow} A \overset{g}{\rightarrow}B$ be a sequence of ring morphisms and let $I$ be an ideal of $A$. Then: $$\dfrac{\widetilde{A}_{B,R}}{I\widetilde{A}_{B,R}}\cong \widetilde{\left(\dfrac{A}{I}\right)}_{\frac{B}{IB},R}$$
		
		\noindent through an isomorphism which takes $u+I\widetilde{A}_{B,R}\mapsto u+IB, \forall u\in \widetilde{A}_{B,R}$.
	\end{prop}
	
	\begin{proof}
		Consider the diagram % https://q.uiver.app/#q=WzAsNixbMCwwLCJSIl0sWzIsMCwiQSJdLFs0LDAsIkIiXSxbMCwyLCJSIl0sWzIsMiwiXFxkZnJhY3tBfXtJfSJdLFs0LDIsIlxcZGZyYWN7Qn17SUJ9Il0sWzAsMSwiXFx0YXUiXSxbMSwyLCJnIl0sWzMsNCwiXFxwaVxcY2lyY1xcdGF1IiwyXSxbNCw1LCJcXGJhcntnfSIsMl0sWzAsMywiXFx0ZXh0cm17aWR9X1IiLDJdLFsxLDQsIlxccGkiXSxbMiw1LCJcXGJhcntcXHBpfSJdXQ==
		\[\begin{tikzcd}
			R && A && B \\
			\\
			R && {\dfrac{A}{I}} && {\dfrac{B}{IB}}
			\arrow["\tau", from=1-1, to=1-3]
			\arrow["g", from=1-3, to=1-5]
			\arrow["\pi\circ\tau"', from=3-1, to=3-3]
			\arrow["{\bar{g}}"', from=3-3, to=3-5]
			\arrow["{\textrm{id}_R}"', from=1-1, to=3-1]
			\arrow["\pi", from=1-3, to=3-3]
			\arrow["{\bar{\pi}}", from=1-5, to=3-5]
		\end{tikzcd} \eqno{(\star)}\]
		
		\noindent where $\pi$ and $\bar{\pi}$ are the quotient maps, and $\bar{g}$ is the canonical morphism induced by $g$ and $I$. It is clear this diagram is commutative, and since $\pi$ is surjective, Proposition \ref{202306011826} ensures that $(\star)$ is a strong Lipman diagram, hence a Maranesi diagram. In this case, $p:\dfrac{B}{IB}\ten_R\dfrac{B}{IB}\rightarrow \dfrac{B}{IB}\ten_R\dfrac{B}{IB}$ is the identity map, so its kernel is a nil ideal. Besides, the quotient map $\bar{\pi}$ is surjective, so it is integral. By Theorem \ref{202410241734} we have $$\bar{\pi}^{-1}\left(\widetilde{\left(\dfrac{A}{I}\right)}_{\frac{B}{IB},R}\right)=\widetilde{A}_{B,R}.$$ 
		
		Again, once $\bar{\pi}$ is surjective, then the last equation implies that $\bar{\pi}(\widetilde{A}_{B,R})= \widetilde{\left(\dfrac{A}{I}\right)}_{\frac{B}{IB},R}$, and also the morphism $$\begin{matrix}
			\widetilde{A}_{B,R} & \longrightarrow & \widetilde{\left(\dfrac{A}{I}\right)}_{\frac{B}{IB},R}\\
			u & \longmapsto & \bar{\pi}(u)=u+IB
		\end{matrix}$$
		
		\noindent is surjective, whose kernel is $\widetilde{A}_{B,R}\cap IB=I\widetilde{A}_{B,R}$. Therefore, this morphism induces the desired isomorphism.
	\end{proof}
	
	Next, we conclude that the quotient inherits the property to be weakly normal.
	
	\begin{cor}
		Let $R\overset{\tau}{\rightarrow} A \overset{g}{\rightarrow}B$ be a sequence of ring morphisms and let $I$ be an ideal of $A$. If $A$ is $R$-weakly normal in $B$ then $\dfrac{A}{I}$ is $R$-weakly normal in $\dfrac{B}{IB}$.
	\end{cor}
	
	\begin{proof}
		Consider the diagram % https://q.uiver.app/#q=WzAsNixbMCwwLCJSIl0sWzEsMCwiQSJdLFsyLDAsIkIiXSxbMCwxLCJSIl0sWzEsMSwiXFxkZnJhY3tBfXtJfSJdLFsyLDEsIlxcZGZyYWN7Qn17SUJ9Il0sWzAsMSwiXFx0YXUiXSxbMSwyLCJnIl0sWzAsMywiXFx0ZXh0cm17aWR9X1IiLDJdLFsxLDQsIlxccGkiXSxbMiw1LCJcXGJhcntcXHBpfSJdLFszLDQsIlxccGlcXGNpcmNcXHRhdSIsMl0sWzQsNSwiXFxiYXJ7Z30iLDJdXQ==
		\[\begin{tikzcd}
			R & A & B \\
			R & {\dfrac{A}{I}} & {\dfrac{B}{IB}}
			\arrow["\tau", from=1-1, to=1-2]
			\arrow["g", from=1-2, to=1-3]
			\arrow["{\textrm{id}_R}"', from=1-1, to=2-1]
			\arrow["\pi", from=1-2, to=2-2]
			\arrow["{\bar{\pi}}", from=1-3, to=2-3]
			\arrow["\pi\circ\tau"', from=2-1, to=2-2]
			\arrow["{\bar{g}}"', from=2-2, to=2-3]
		\end{tikzcd}\] \noindent and let $\widetilde{\pi}:\dfrac{\widetilde{A}_{B,R}}{I\widetilde{A}_{B,R}}\rightarrow \widetilde{\left(\dfrac{A}{I}\right)}_{\frac{B}{IB},R}$ be the isomorphism obtained in the previous proposition. 
		
		Let us prove that $\widetilde{\left(\dfrac{A}{I}\right)}_{\frac{B}{IB},R}\sub\bar{g}\left(\dfrac{A}{I}\right)$. If $w\in \widetilde{\left(\dfrac{A}{I}\right)}_{\frac{B}{IB},R}$ then we can write $$w=\widetilde{\pi}(u+I\widetilde{A}_{B,R})=u+IB,$$ \noindent for some $u\in \widetilde{A}_{B,R}$. Since $A$ is $R$-weakly normal in $B$ then $\widetilde{A}_{B,R}=g(A)$, and then there exists $a\in A$ such that $u=g(a)$. Hence, $w=g(a)+IB=\bar{\pi}(g(a))=\bar{g}(\pi(a))=\bar{g}(a+I)\in\bar{g}\left(\dfrac{A}{I}\right)$. 
		
		Therefore, $\dfrac{A}{I}$ is $R$-weakly normal in $\dfrac{B}{IB}$.
	\end{proof}

	Next, we want to prove an analogous result to \cite[Proposition 1.4]{L} for the weak normalization. First, we prove an auxiliary lemma.
	
	\begin{lema}\label{202410262158}
		Let $h: S\rightarrow T$ be a ring morphism and suppose that $\ker h$ is a nil ideal of $S$. If $x\in S$ and $h(x)\in\sqrt{(0_T)}$ then $x\in\sqrt{(0_S)}$.
	\end{lema}
	
	\begin{proof}
		By hypothesis there exists $r\in\bN$ such that $(h(x))^r=0_T$, i.e., $h(x^r)=0_T$. Thus, $x^r\in\ker h\sub\sqrt{(0_S)}$ and this implies that $x\in\sqrt{\sqrt{(0_S)}}=\sqrt{(0_S)}$.
	\end{proof}
	
	\begin{prop}
		Consider a sequence of ring morphisms $R\longrightarrow A\overset{g}{\longrightarrow}B$ and suppose that $g$ is an integral morphism. Then $\widetilde{A}_{B,R}$ is a radicial $A$-algebra.
	\end{prop}
	
	\begin{proof}
		We need to check that the kernel of the canonical morphism $\gamma:\widetilde{A}_{B,R}\ten_A\widetilde{A}_{B,R}\rightarrow \widetilde{A}_{B,R}$ is a nil ideal of $\widetilde{A}_{B,R}\ten_A\widetilde{A}_{B,R}$. Since $\ker\gamma$ is generated by $\{x\ten_A1_B-1_B\ten_Ax\mid x\in\widetilde{A}_{B,R}\mid x\in\widetilde{A}_{B,R}\}$, it suffices to show that $$\delta(x):=x\ten_A1_B-1_B\ten_Ax\in\sqrt{\ideal{0_{\widetilde{A}_{B,R}\ten_A\widetilde{A}_{B,R}}}}, \forall x\in \widetilde{A}_{B,R}.$$ 
		
		Observe that, by hypothesis, $B$ is integral over $g(A)$, and since $g(A)\sub\widetilde{A}_{B,R}\sub B$ then $B$ is integral over $\widetilde{A}_{B,R}$. Consequently, the inclusion $\iota:\widetilde{A}_{B,R}\hookrightarrow B$ is an integral morphism of $A$-algebras. By Lemma \ref{202410241420}, we can conclude that the kernel of the ring morphism $\iota\ten_A\iota:\widetilde{A}_{B,R}\ten_A\widetilde{A}_{B,R}\rightarrow B\ten_AB$ is a nil ideal of $\widetilde{A}_{B,R}\ten_A\widetilde{A}_{B,R}$. 
		
		Finally, let us check what remains. Let $x\in\widetilde{A}_{B,R}$. So, $\Delta(x)\in\sqrt{\ker\varphi}$, i.e., there exists $n\in\bN$ such that $(\Delta(x))^n=0_{B\ten_RB}$, and thus $(\varphi(\Delta(x)))^n=0_{B\ten_AB}$. It is easy to see that $(\iota\ten_A\iota)(\delta(x))=\varphi(\Delta(x))$, hence $(\iota\ten_A\iota)(\delta(x))\in\sqrt{(0_{B\ten_AB})}$. Since $\ker(\iota\ten_A\iota)$ is a nil ideal then Lemma \ref{202410262158} ensures that $\delta(x)\in\sqrt{0_{\widetilde{A}_{B,R}\ten_A\widetilde{A}_{B,R}}}$, which finishes the proof.		
	\end{proof}

	\section{Radicial and unramified contraction}
	
	In \cite{L} Lipman proved the following proposition.
	
	\begin{prop}\cite[Proposition 1.3]{L}\label{202410262346}
		Supppose that $R\overset{f_R}{\longrightarrow}R'\longrightarrow A\overset{g}{\longrightarrow} B$ is a sequence of ring morphisms. If $f_R$ is radicial or unramified then $A^*_{B,R}=A^*_{B,R'}$. 
	\end{prop}
	
	Next, we prove an analogous result to Theorems \ref{202410241447} and \ref{202410241734} combining radiciality of $f_R$ with integrality of $f$, and this result will generalize Proposition \ref{202410262346} for the relative Lipschitz saturation in the case of radicial algebras.
	
	\begin{teo}[\textbf{Radicial + Integral} contraction]\label{202410262329}
		Consider the following commutative diagram of ring morphisms % https://q.uiver.app/#q=WzAsNixbMCwwLCJSIl0sWzEsMCwiQSJdLFsyLDAsIkIiXSxbMCwxLCJSJyJdLFsxLDEsIkEnIl0sWzIsMSwiQiciXSxbMCwzLCJmX1IiLDJdLFsxLDQsImZfQSIsMl0sWzIsNSwiZiIsMl0sWzAsMSwiXFx0YXUiXSxbMSwyLCJnIl0sWzMsNCwiXFx0YXUnIiwyXSxbNCw1LCJnJyIsMl1d
		\[\begin{tikzcd}
			R & A & B \\
			{R'} & {A'} & {B'}
			\arrow["\tau", from=1-1, to=1-2]
			\arrow["{f_R}"', from=1-1, to=2-1]
			\arrow["g", from=1-2, to=1-3]
			\arrow["{f_A}"', from=1-2, to=2-2]
			\arrow["f"', from=1-3, to=2-3]
			\arrow["{\tau'}"', from=2-1, to=2-2]
			\arrow["{g'}"', from=2-2, to=2-3]
		\end{tikzcd}\eqno{(\star)}\]
		
		\noindent and suppose that $f_R$ is radicial and $f$ is integral.
		
		\begin{enumerate}
			
			\item [a)] If $(\star)$ is a Lipman diagram then $f^{-1}((A')^*_{B',R'})=A^*_{B,R}$;
			
			\item [b)] If $(\star)$ is a Maranesi diagram then $f^{-1}(\widetilde{A'}_{B', R'})=\widetilde{A}_{B,R}$.
		\end{enumerate}
	\end{teo}
	
	\begin{proof}
		Since $f_R$ is radicial then by Lemma \ref{202410312042}  $\ker p$ is a nil ideal of $B'\ten_RB'$. Therefore, item (a) is a consequence of Theorem \ref{202410262320}, and item (b) follows by Theorem \ref{202410241734}.	
	\end{proof}
	
	\begin{cor}\label{202410262342}
		If $R\overset{f_R}{\longrightarrow}R'\longrightarrow A\overset{g}{\longrightarrow} B$ is a sequence of ring morphisms and $f_R$ is radicial then: 
		
		\begin{enumerate}
			\item [a)] $A^*_{B,R}=A^*_{B,R'}$;
			
			\item [b)] $\widetilde{A}_{B,R}=\widetilde{A}_{B,R'}$.

		\end{enumerate}
	\end{cor}
	
	\begin{proof}
		It is a immediate consequence of Theorem \ref{202410262329}, since % https://q.uiver.app/#q=WzAsNixbMCwwLCJSIl0sWzEsMCwiQSJdLFsyLDAsIkIiXSxbMCwxLCJSJyJdLFsxLDEsIkEiXSxbMiwxLCJCIl0sWzAsMywiZl9SIiwyXSxbMSw0LCJcXHRleHRybXtpZH1fQSIsMl0sWzIsNSwiXFx0ZXh0cm17aWR9X0IiLDJdLFswLDEsIlxcdGF1Il0sWzEsMiwiZyJdLFszLDQsIlxcdGF1JyIsMl0sWzQsNSwiZyIsMl1d
		\[\begin{tikzcd}
			R & A & B \\
			{R'} & A & B
			\arrow[from=1-1, to=1-2]
			\arrow["{f_R}"', from=1-1, to=2-1]
			\arrow["g", from=1-2, to=1-3]
			\arrow["{\textrm{id}_A}"', from=1-2, to=2-2]
			\arrow["{\textrm{id}_B}"', from=1-3, to=2-3]
			\arrow[from=2-1, to=2-2]
			\arrow["g"', from=2-2, to=2-3]
		\end{tikzcd}\] is a strong Lipman diagram, hence, a Lipman/Maranesi diagram.
	\end{proof}
	
	\begin{cor}
		Let $R$ be a ring and  $Q(R)$ be the total quotient ring of $R$. If $g:A\rightarrow B$ is a morphism of algebras over $Q(R)$ then: 
		
		\begin{enumerate}
			\item [a)] $A^*_{B,R}=A^*_{B, Q(R)}$;
			
			\item [b)] $\widetilde{A}_{B,R}=\widetilde{A}_{B, Q(R)}$.

		\end{enumerate}

	\end{cor}
	
	\begin{proof}
		If $S$ is the subset of elements of $R$ which are not zero divisors in $R$, by definition $Q(R)=S^{-1}R$. In Example \ref{202411172308}, we have seen that the localization map $\rho:R\rightarrow Q(R)$ is a radicial ring morphism. Hence, this result follows by Corollary \ref{202410262342}.
	\end{proof}
	
	Notice that Corollary \ref{202410262342} (a) was obtained by Lipman in Proposition \ref{202410262346}, and in item (b) we obtained its analogous version for the relative weak normalization. Also, item (b) was already obtained in Corollary \ref{202410251558} using another strategy.
	
	As a consequence, Corollary \ref{202410262342} (b) allows us to obtain an analogous result to \cite[Proposition 3.8]{SR} for the weak normalization.
	
	Consider again the sequence of ring morphism $R\overset{\tau}{\rightarrow} A \overset{g}{\rightarrow}B$ where $I$ is an ideal of $A$. In comparison with the diagram used in Proposition \ref{202410251529}, we can form another type of diagram % https://q.uiver.app/#q=WzAsNixbMCwwLCJSIl0sWzIsMCwiQSJdLFs0LDAsIkIiXSxbMCwyLCJcXGRmcmFje1J9e1xcdGF1XnstMX0oSSl9Il0sWzIsMiwiXFxkZnJhY3tBfXtJfSJdLFs0LDIsIlxcZGZyYWN7Qn17SUJ9Il0sWzAsMSwiXFx0YXUiXSxbMSwyLCJnIl0sWzMsNCwiXFxiYXJ7XFx0YXV9IiwyXSxbNCw1LCJcXGJhcntnfSIsMl0sWzAsMywiXFxwaV9SIiwyXSxbMSw0LCJcXHBpIl0sWzIsNSwiXFxiYXJ7XFxwaX0iXV0=
	\[\begin{tikzcd}
		R && A && B \\
		\\
		{\dfrac{R}{\tau^{-1}(I)}} && {\dfrac{A}{I}} && {\dfrac{B}{IB}}
		\arrow["\tau", from=1-1, to=1-3]
		\arrow["g", from=1-3, to=1-5]
		\arrow["{\bar{\tau}}"', from=3-1, to=3-3]
		\arrow["{\bar{g}}"', from=3-3, to=3-5]
		\arrow["{\pi_R}"', from=1-1, to=3-1]
		\arrow["\pi", from=1-3, to=3-3]
		\arrow["{\bar{\pi}}", from=1-5, to=3-5]
	\end{tikzcd}\].
	
	Next, we show that we can replace $R$ by $\dfrac{R}{\tau^{-1}(I)}$ on the right-hand side of the conclusion of Proposition \ref{202410251529}.
	
	\begin{prop}
		With the above notation, $$\widetilde{\left(\dfrac{A}{I}\right)}_{\frac{B}{IB},\frac{R}{\tau^{-1}(I)}}=\widetilde{\left(\dfrac{A}{I}\right)}_{\frac{B}{IB},R}\cong\dfrac{\widetilde{A}_{B,R}}{I\widetilde{A}_{B,R}}$$
	\end{prop}
	
	\begin{proof}
		In this case the canonical morphism $\gamma:\dfrac{R}{\tau^{-1}(I)}\ten_R\dfrac{R}{\tau^{-1}(I)}\rightarrow \dfrac{R}{\tau^{-1}(I)}$ is an isomorphism, so its kernel is a nil ideal. Thus\footnote{See Prop. 3.7.1, p. 246 in \cite{EGA}} $\dfrac{R}{\tau^{-1}(I)}$ is a radicial $R$-algebra (with the canonical structure). From the sequence % https://q.uiver.app/#q=WzAsNCxbMCwwLCJSIl0sWzEsMCwiXFxkZnJhY3tSfXtcXHRhdV57LTF9KEkpfSJdLFsyLDAsIlxcZGZyYWN7QX17SX0iXSxbMywwLCJcXGRmcmFje0J9e0lCfSJdLFswLDFdLFsxLDIsIlxcYmFye1xcdGF1fSJdLFsyLDMsIlxcYmFye2d9Il1d
		\[\begin{tikzcd}
			R & {\dfrac{R}{\tau^{-1}(I)}} & {\dfrac{A}{I}} & {\dfrac{B}{IB}}
			\arrow[from=1-1, to=1-2]
			\arrow["{\bar{\tau}}", from=1-2, to=1-3]
			\arrow["{\bar{g}}", from=1-3, to=1-4]
		\end{tikzcd},\]
		
		\noindent Corollary \ref{202410262342} (b) ensures that $\widetilde{\left(\dfrac{A}{I}\right)}_{\frac{B}{IB},\frac{R}{\tau^{-1}(I)}}=\widetilde{\left(\dfrac{A}{I}\right)}_{\frac{B}{IB},R}$.
	\end{proof}
	
	Next, we emphasize that the contraction property is obtained without any \textit{a priori} requirement about the involved diagram to be Maranesi since the hypotheses will ensure it automatically.
	
	\begin{cor}[\textbf{Radicial + Surjective} contraction]\label{202411031653} 	Suppose that
		\[\begin{tikzcd}
			R & A & B \\
			{R'} & {A'} & {B'}
			\arrow["\tau", from=1-1, to=1-2]
			\arrow["g", from=1-2, to=1-3]
			\arrow["{\tau'}", from=2-1, to=2-2]
			\arrow["{g'}", from=2-2, to=2-3]
			\arrow["{f_R}"', from=1-1, to=2-1]
			\arrow["{f}", from=1-3, to=2-3]
			\arrow["{f_A}", from=1-2, to=2-2]
		\end{tikzcd}\eqno{(\star)}\]
		
		\noindent is a commutative diagram of ring morphisms. If $f_R$ and $f_A$ are radicial and $f$ is surjective then $$f^{-1}(\widetilde{A'}_{B',R'})=\widetilde{A}_{B,R}.$$
	\end{cor}
	
	\begin{proof}
		Since $f_A$ is radicial and $f$ is surjective then Proposition \ref{202411031659} ensures that $(\star)$ is a Maranesi diagram. Again, since $f$ is surjective then $f$ is integral, and once $f_R$ is radicial, the result now is a direct consequence of Theorem \ref{202410262329}.
	\end{proof}

	Now we want to obtain for the weak normalization a result analogous to Proposition \ref{202410262346} for the unramified condition. In \cite{L}, Lipman proved the following result.
	
	\begin{lema}\cite[Lemma 1.1]{L}\label{202410262229}
		Let $R$ be a ring and let $J_1,\hdots,J_n, I$ ideals of $R$ such that $J_1\cdots J_n$ is a nil ideal of $R$. If $x+J_i\in\overline{I\frac{R}{J_i}}, \forall i\in\{1,\hdots,n\}$, then $x\in\overline{I}$.
	\end{lema}
	
	For the weak normalization, we prove an analogous lemma for the radical of ideals.
	
	\begin{lema}\label{202410262230}
		Let $R$ be a ring and let $J_1,\hdots,J_n, I$ ideals of $R$ such that $J_1\cdots J_n$ is a nil ideal of $R$. If $x+J_i\in\sqrt{I\frac{R}{J_i}}, \forall i\in\{1,\hdots,n\}$, then $x\in\sqrt{I}$.
	\end{lema}
	
	\begin{proof}
		By hypothesis, for each $i\in\{1,\hdots,n\}$ there exists $m_i\in\bN$ such that $(x+J_i)^{m_i}\in I\frac{R}{J_i}$, i.e., $x^{m_i}+J_i\in I\frac{R}{J_i}$, and we can write $$x^{m_i}+J_i=a_{i1}(\alpha_{i1}+J_i)+\cdots+a_{ir_i}(\alpha_{ir_i}+J_i),$$
		\noindent where $a_{i1},\hdots, a_{ir_i}\in I$ and $\alpha_{i1},\hdots, \alpha_{ir_i}\in R$. Thus, for each $i\in\{1,\hdots,n\}$, taking $$x_i:=a_{i1}\alpha_{i1}+\cdots+a_{ir_i}\alpha_{ir_i}\in I,$$ \noindent  we have $x^{m_i}-x_i\in J_i$. Hence, $(x^{m_1}-x_1)\cdots(x^{m_n}-x_n)\in J_1\cdots J_n\sub\sqrt{(0_R)}$, and this implies that there exists $r\in\bN$ such that $$((x^{m_1}-x_1)\cdots(x^{m_n}-x_n))^r=0_R.$$
		
		Since $x_1,\hdots,x_n\in I$ then the last equation implies that $x^{(m_1+\cdots+m_n)r}\in I$. Therefore, $x\in\sqrt{I}$.
	\end{proof}

	Now we state the main theorem involving diagrams with the first morphism of algebras being unramified. We emphasize that we do not need to require \textit{a priori} that the diagram is Maranesi.

	\begin{teo}[\textbf{Unramified + Radicial + Surjective} WN contraction]\label{202410271947}
		Let % https://q.uiver.app/#q=WzAsNixbMCwwLCJSIl0sWzEsMCwiQSJdLFsyLDAsIkIiXSxbMCwxLCJSJyJdLFsxLDEsIkEnIl0sWzIsMSwiQiciXSxbMCwxLCJcXHRhdSJdLFsxLDIsImciXSxbMiw1LCJmIl0sWzEsNCwiZl9BIl0sWzAsMywiZl9SIiwyXSxbMyw0LCJcXHRhdSciLDJdLFs0LDUsImcnIiwyXV0=
		\[\begin{tikzcd}
			R & A & B \\
			{R'} & {A'} & {B'}
			\arrow["\tau", from=1-1, to=1-2]
			\arrow["{f_R}"', from=1-1, to=2-1]
			\arrow["g", from=1-2, to=1-3]
			\arrow["{f_A}", from=1-2, to=2-2]
			\arrow["f", from=1-3, to=2-3]
			\arrow["{\tau'}"', from=2-1, to=2-2]
			\arrow["{g'}"', from=2-2, to=2-3]
		\end{tikzcd}\eqno{(\star)}\] be a commutative diagram of ring morphisms and suppose that: 
		
		\begin{itemize}
			\item $f_R$ is unramified;
			
			\item $f_A$ is radicial;
			
			\item $f$ is surjective. 
		\end{itemize}
		
		Then: $$f^{-1}(\widetilde{A'}_{B',R'})=\widetilde{A}_{B,R}.$$
	\end{teo}

	\begin{proof}
		
		Firstly, we recall the hypotheses that allow us to apply Proposition \ref{202411031659} to conclude that $(\star)$ is a Maranesi diagram. Let $\gamma: R'\ten_RR'$ be the canonical morphism. Since $f_R$ is unramified then there exists $e'\in\ker\gamma$ such that $e'^2=e'$ and $\ker\gamma=\ideal{e'}$. Consider the standard commutative diagram

		% https://q.uiver.app/#q=WzAsMTUsWzAsMCwiUiJdLFsxLDAsIkEiXSxbMiwwLCJCIl0sWzAsMSwiUiciXSxbMSwxLCJBJyJdLFsyLDEsIkInIl0sWzAsMiwiUiciXSxbMywwLCJCXFxvdGltZXNfUkIiXSxbMywxLCJCJ1xcb3RpbWVzX1JCJyJdLFs0LDAsIkJcXG90aW1lc19BQiJdLFsxLDIsIkEnIl0sWzIsMiwiQiciXSxbMywyLCJCJ1xcb3RpbWVzX3tSJ31CJyJdLFs0LDIsIkInXFxvdGltZXNfe0EnfUIiXSxbNCwxLCJCJ1xcb3RpbWVzX0FCJyJdLFswLDMsImZfUiIsMl0sWzEsNCwiZl9BIiwyXSxbMiw1LCJmIiwyXSxbMCwxLCJcXHRhdSJdLFsxLDIsImciXSxbMyw0LCJcXHRhdSciLDJdLFs0LDUsImcnIiwyXSxbMiw3LCJcXERlbHRhIl0sWzcsOSwiXFx2YXJwaGkiXSxbNSw4LCJcXHRpbGRle1xcRGVsdGF9IiwyXSxbNiwxMCwiXFx0YXUnIiwyXSxbMyw2LCJcXHRleHRybXtpZH1fe1InfSIsMl0sWzQsMTAsIlxcdGV4dHJte2lkfV97QSd9IiwyXSxbMTAsMTEsImcnIiwyXSxbNSwxMSwiXFx0ZXh0cm17aWR9X3tCJ30iLDJdLFsxMSwxMiwiXFxEZWx0YSciLDJdLFs4LDEyLCJwIiwyXSxbNyw4LCJmXFxvdGltZXNfUmYiLDJdLFsxMiwxMywiXFx2YXJwaGknIiwyXSxbOCwxNF0sWzksMTQsImZcXG90aW1lc19BZiJdLFsxNCwxMywicF9BIl0sWzcsMTIsIlxcb3ZlcmxpbmV7XFxzY3JpcHRzaXple2Z9fSIsMCx7ImxhYmVsX3Bvc2l0aW9uIjoxMCwib2Zmc2V0IjotMSwiY3VydmUiOi01LCJzdHlsZSI6eyJib2R5Ijp7Im5hbWUiOiJkYXNoZWQifX19XV0=
		\[\begin{tikzcd}
			R & A & B & {B\otimes_RB} & {B\otimes_AB} \\
			{R'} & {A'} & {B'} & {B'\otimes_RB'} & {B'\otimes_AB'} \\
			{R'} & {A'} & {B'} & {B'\otimes_{R'}B'} & {B'\otimes_{A'}B}
			\arrow["\tau", from=1-1, to=1-2]
			\arrow["{f_R}"', from=1-1, to=2-1]
			\arrow["g", from=1-2, to=1-3]
			\arrow["{f_A}"', from=1-2, to=2-2]
			\arrow["\Delta", from=1-3, to=1-4]
			\arrow["f"', from=1-3, to=2-3]
			\arrow["\varphi", from=1-4, to=1-5]
			\arrow["{f\otimes_Rf}"', from=1-4, to=2-4]
			\arrow["{\overline{\scriptsize{f}}}"{pos=0.3}, shift left, curve={height=-30pt}, dashed, from=1-4, to=3-4]
			\arrow["{f\otimes_Af}", from=1-5, to=2-5]
			\arrow["{\tau'}"', from=2-1, to=2-2]
			\arrow["{\textrm{id}_{R'}}"', from=2-1, to=3-1]
			\arrow["{g'}"', from=2-2, to=2-3]
			\arrow["{\textrm{id}_{A'}}"', from=2-2, to=3-2]
			\arrow["{\tilde{\Delta}}"', from=2-3, to=2-4]
			\arrow["{\textrm{id}_{B'}}"', from=2-3, to=3-3]
			\arrow[from=2-4, to=2-5]
			\arrow["p"', from=2-4, to=3-4]
			\arrow["{p_A}", from=2-5, to=3-5]
			\arrow["{\tau'}"', from=3-1, to=3-2]
			\arrow["{g'}"', from=3-2, to=3-3]
			\arrow["{\Delta'}"', from=3-3, to=3-4]
			\arrow["{\varphi'}"', from=3-4, to=3-5]
		\end{tikzcd}\eqno{(\star)}\]
		
		We know that $\ker p=\ideal{\mbox{Im}(\tilde{\Delta}\circ g'\circ \tau')}$. If $\sigma:R'\ten_RR'\rightarrow B'\ten_R B'$ is given by $\sigma:=(g'\ten_Rg')\circ(\tau'\ten_R\tau')$, it is easy to see that $$\tilde{\Delta}\circ g'\circ\tau'(r')=\sigma(r'\ten_R1_{R'}-1_{R'}\ten_Rr'), \forall r'\in R'.$$
		
		Thus, $\ker p=\ideal{\sigma(\ker\gamma)}=\ideal{e}$, where $e:=\sigma(e')$. Notice that $e^2=e$. Since $f\ten_Rf$ is surjective then there exists $e_0\in B\ten_RB$ such that $e=(f\ten_Rf)(e_0)$. Since $e\in\ker p$ then $p((f\ten_Rf)(e_0))=p(e)=0_{B'\ten_{R'}B'}$, i.e.,  $\overline{f}(e_0)=0_{B'\ten_{R'}B'}\implies e_0\in\ker\overline{f}$.
		
		We already have $\widetilde{A}_{B,R}\sub f^{-1}(\widetilde{A'}_{B',R'})$. Conversely, let $x\in f^{-1}(\widetilde{A'}_{B',R'})$. Thus, $f(x)\in \widetilde{A'}_{B',R'}$, i.e., $\overline{f}(\Delta(x))=\Delta'(f(x))\in\sqrt{\ker\varphi'}$. As we have seen before, $(\star)$ is a Maranesi diagram, and this implies $\sqrt{\ker\varphi'}=\sqrt{\ker\varphi(B'\ten_{R'}B')}$. Further, since $f\ten_Rf$ and $p$ are surjective then $\overline{f}=p\circ(f\ten_Rf)$ is surjective, and consequently, $\ker\varphi(B'\ten_{R'}B')=\overline{f}(\ker\varphi)$. So, $\overline{f}(\Delta(x))\in\sqrt{\overline{f}(\ker\varphi)}$, which implies there exist $s\in\bN$ and $d\in\ker\varphi$ such that: $$\overline{f}((\Delta(x))^s)=\overline{f}(d)\implies (f\ten_Rf)((\Delta(x))^s-d)\in\ker p=\ideal{(f\ten_Rf)(e_0)},$$
		
		\noindent and using again the surjectivity of $f\ten_Rf$, we can find $w\in B\ten_RB$ such that $$(f\ten_Rf)((\Delta(x))^s-d)=(f\ten_Rf)(we_0)\implies \Delta(x))^s-d-we_0\in\ker(f\ten_Rf).$$
		
		If we take the ideal of $B\ten_RB$ given by $J_1:=\ker(f\ten_Rf)+\ideal{e_0}$, one has that $\Delta(x))^s-d\in J_1$, and consequently, $$\Delta(x)+J_1\in\sqrt{\ker\varphi\dfrac{B\ten_RB}{J_1}}. \eqno{(1)}.$$
		
		We have seen that $\overline{f}(e_0)=0_{B'\ten_{R'}B'}$, and this implies $\varphi'(\overline{f}(e_0))=0_{B'\ten_{A'}B'}$. Since $(\star)$ is a commutative diagram then $p_A((f\ten_Af)(\varphi(e_0)))=0_{B'\ten_{A'}B'}$. Since $f_A$ is radicial, Lemma \ref{202410312042} guarantees that $\ker p_A\sub\sqrt{\ideal{0_{B'\ten_AB'}}}$. Thus, $$(f\ten_Af)(\varphi(e_0))\in\ker p_A\sub \sqrt{\ideal{0_{B'\ten_AB'}}}\implies \varphi(e_0)\in\sqrt{\ker(f\ten_Af)}.$$
		
		Since $f$ is an integral morphism (once it is surjective), Lemma \ref{202410241420} (a) guarantees that $$\ker(f\ten_Rf)\sub\sqrt{\ideal{0_{B\ten_{R}B}}}\mbox{ and }\ker(f\ten_Af)\sub\sqrt{\ideal{0_{B\ten_{A}B}}}.$$
		
		Hence, $\varphi(e_0)\in\sqrt{\ideal{0_{B\ten_AB}}}$ and, consequently, there exists $t\in\bN$ such that $\varphi(e_0^t)=0_{B\ten_AB}$, i.e., $e_0^t\in\ker\varphi$. Denote $\mathbf{1}:=1_{B\ten_RB}$. Consider the ideal of $B\ten_RB$ given by $J_2:=\ideal{\mathbf{1}-e_0}$, and take $c:=\Delta(x)e_0$. Observe that $c^t=(\Delta(x))^t\underbrace{e_0^t}_{\in\ker\varphi}\in\ker\varphi$. Thus, $$(\Delta(x))^t-c^t=(\Delta(x))^t(\mathbf{1}-e_0^t)=(\Delta(x))^t(\mathbf{1}-e_0)(\mathbf{1}+e_0+\cdots+e_0^{t-1})\in J_2.$$
		
		Hence, $(\Delta(x))^t+J_2=\underbrace{c^t}_{\in\ker\varphi}+J_2$, i.e., $$\Delta(x)+J_2\in\sqrt{\ker\varphi\dfrac{B\ten_RB}{J_2}}. \eqno{(2)}$$
		
		Moreover, since $e^2=e$ and $e=(f\ten_Rf)(e_0)$ then $e_0-e_0^2\in\ker(f\ten_Rf)$. Thus, $$J_1J_2=(\ker(f\ten_Rf)+\ideal{e_0})\ideal{\mathbf{1}-e_0}=\ker(f\ten_Rf)\ideal{\mathbf{1}-e_0}+\ideal{e_0}\ideal{\mathbf{1}-e_0}$$$$\sub \ker(f\ten_Rf)+\underbrace{\ideal{e_0-e_0^2}}_{\sub \ker(f\ten_Rf)}\sub \ker(f\ten_Rf)\sub \sqrt{\ideal{0_{B\ten_RB}}}.$$
		
		Hence, $J_1J_2$ is a nil ideal of $B\ten_RB$, and using (1), (2) and Lemma \ref{202410262230}, we can conclude that $\Delta(x)\in \sqrt{\ker\varphi}$. Therefore, $x\in \widetilde{A}_{B,R}$.
	\end{proof}

	As a consequence, we obtain a similar result to Proposition \ref{202410262346} for the weak normalization in the context of unramified algebras.
	
	\begin{cor}
		Consider the sequence of ring morphisms % https://q.uiver.app/#q=WzAsNCxbMSwwLCJSJyJdLFsyLDAsIkEiXSxbMywwLCJCIl0sWzAsMCwiUiJdLFswLDEsIlxcdGF1JyJdLFsxLDIsImciXSxbMywwLCJmX1IiXSxbMywxLCJcXHRhdSIsMix7ImN1cnZlIjoyfV1d
		\[\begin{tikzcd}
			R & {R'} & A & B
			\arrow["{f_R}", from=1-1, to=1-2]
			\arrow["\tau"', curve={height=24pt}, from=1-1, to=1-3]
			\arrow["{\tau'}", from=1-2, to=1-3]
			\arrow["g", from=1-3, to=1-4]
		\end{tikzcd}\]
		
		If $f_R$ is unramified then $\widetilde{A}_{B,R}=\widetilde{A}_{B,R'}$.
	\end{cor}
	
	\begin{proof}
		
		Notice that % https://q.uiver.app/#q=WzAsNixbMCwwLCJSIl0sWzAsMSwiUiciXSxbMSwwLCJBIl0sWzEsMSwiQSJdLFsyLDAsIkIiXSxbMiwxLCJCIl0sWzAsMiwiXFx0YXUiXSxbMiw0LCJnIl0sWzQsNSwiXFx0ZXh0cm17aWR9X0IiXSxbMiwzLCJcXHRleHRybXtpZH1fQSJdLFswLDEsImZfUiIsMl0sWzEsMywiXFx0YXUnIiwyXSxbMyw1LCJnIiwyXV0=
		$\begin{tikzcd}
			R & A & B \\
			{R'} & A & B
			\arrow["\tau", from=1-1, to=1-2]
			\arrow["{f_R}"', from=1-1, to=2-1]
			\arrow["g", from=1-2, to=1-3]
			\arrow["{\textrm{id}_A}", from=1-2, to=2-2]
			\arrow["{\textrm{id}_B}", from=1-3, to=2-3]
			\arrow["{\tau'}"', from=2-1, to=2-2]
			\arrow["g"', from=2-2, to=2-3]
		\end{tikzcd}$ is a Maranesi diagram, where $f_R$ is unramified, $\id_A$ is radicial and $\id_B$ is surjective.\end{proof}
	
	\begin{lema}\label{202410280018}
		Let $h: S\rightarrow T$ be an integral morphism and let $I$ be an ideal of $S$. If $x\in h^{-1}(\overline{IT})$ then there exist $a_i\in I^i, i\in\{1,\hdots,n\}$ such that $$a_n+a_{n-1}x+\cdots+a_1x^{n-1}+x^n\in\ker h.$$
	\end{lema}
	
	\begin{proof}
		It is a direct consequence of the proof of Lemma \ref{202410241420} (see the proof in \cite[Lemma 3.2]{SR}).
	\end{proof}
	
	Next, we generalize the \textit{unramified part} of Proposition \ref{202410262346} for Lipman diagrams.\footnote{Here, we emphasize that, unlike in the case of weak normalization, where the hypotheses a priori ensured that the diagram was Maranesi, in the case of Lipschitz saturation, we do not know if this is true or not. Therefore, at this point, we need to require a priori that the diagram is Lipman.}
	
	\begin{teo}[\textbf{Unramified + Radicial + Surjective} Lipschitz contraction]\label{202410280007}
		Let % https://q.uiver.app/#q=WzAsNixbMCwwLCJSIl0sWzEsMCwiQSJdLFsyLDAsIkIiXSxbMCwxLCJSJyJdLFsxLDEsIkEnIl0sWzIsMSwiQiciXSxbMCwxLCJcXHRhdSJdLFsxLDIsImciXSxbMiw1LCJmIl0sWzEsNCwiZl9BIl0sWzAsMywiZl9SIiwyXSxbMyw0LCJcXHRhdSciLDJdLFs0LDUsImcnIiwyXV0=
		\[\begin{tikzcd}
			R & A & B \\
			{R'} & {A'} & {B'}
			\arrow["\tau", from=1-1, to=1-2]
			\arrow["{f_R}"', from=1-1, to=2-1]
			\arrow["g", from=1-2, to=1-3]
			\arrow["{f_A}", from=1-2, to=2-2]
			\arrow["f", from=1-3, to=2-3]
			\arrow["{\tau'}"', from=2-1, to=2-2]
			\arrow["{g'}"', from=2-2, to=2-3]
		\end{tikzcd}\] be a Lipman diagram and suppose that: 
		
		\begin{itemize}
			\item $f_R$ is unramified;
			
			\item $f_A$ is radicial;
			
			\item $f$ is surjective.
		\end{itemize}
		
		Then: $$f^{-1}((A')^*_{B',R'})=A^*_{B,R}.$$
	\end{teo}

	\begin{proof}
		
		Let $\gamma: R'\ten_RR'$ be the canonical morphism. Since $f_R$ is unramified then there exists $e'\in\ker\gamma$ such that $e'^2=e'$ and $\ker\gamma=\ideal{e'}$. Consider the standard commutative diagram

		% https://q.uiver.app/#q=WzAsMTUsWzAsMCwiUiJdLFsxLDAsIkEiXSxbMiwwLCJCIl0sWzAsMSwiUiciXSxbMSwxLCJBJyJdLFsyLDEsIkInIl0sWzAsMiwiUiciXSxbMywwLCJCXFxvdGltZXNfUkIiXSxbMywxLCJCJ1xcb3RpbWVzX1JCJyJdLFs0LDAsIkJcXG90aW1lc19BQiJdLFsxLDIsIkEnIl0sWzIsMiwiQiciXSxbMywyLCJCJ1xcb3RpbWVzX3tSJ31CJyJdLFs0LDIsIkInXFxvdGltZXNfe0EnfUIiXSxbNCwxLCJCJ1xcb3RpbWVzX0FCJyJdLFswLDMsImZfUiIsMl0sWzEsNCwiZl9BIiwyXSxbMiw1LCJmIiwyXSxbMCwxLCJcXHRhdSJdLFsxLDIsImciXSxbMyw0LCJcXHRhdSciLDJdLFs0LDUsImcnIiwyXSxbMiw3LCJcXERlbHRhIl0sWzcsOSwiXFx2YXJwaGkiXSxbNSw4LCJcXHRpbGRle1xcRGVsdGF9IiwyXSxbNiwxMCwiXFx0YXUnIiwyXSxbMyw2LCJcXHRleHRybXtpZH1fe1InfSIsMl0sWzQsMTAsIlxcdGV4dHJte2lkfV97QSd9IiwyXSxbMTAsMTEsImcnIiwyXSxbNSwxMSwiXFx0ZXh0cm17aWR9X3tCJ30iLDJdLFsxMSwxMiwiXFxEZWx0YSciLDJdLFs4LDEyLCJwIiwyXSxbNyw4LCJmXFxvdGltZXNfUmYiLDJdLFsxMiwxMywiXFx2YXJwaGknIiwyXSxbOCwxNF0sWzksMTQsImZcXG90aW1lc19BZiJdLFsxNCwxMywicF9BIl0sWzcsMTIsIlxcb3ZlcmxpbmV7XFxzY3JpcHRzaXple2Z9fSIsMCx7ImxhYmVsX3Bvc2l0aW9uIjoxMCwib2Zmc2V0IjotMSwiY3VydmUiOi01LCJzdHlsZSI6eyJib2R5Ijp7Im5hbWUiOiJkYXNoZWQifX19XV0=
		\[\begin{tikzcd}
			R & A & B & {B\otimes_RB} & {B\otimes_AB} \\
			{R'} & {A'} & {B'} & {B'\otimes_RB'} & {B'\otimes_AB'} \\
			{R'} & {A'} & {B'} & {B'\otimes_{R'}B'} & {B'\otimes_{A'}B}
			\arrow["\tau", from=1-1, to=1-2]
			\arrow["{f_R}"', from=1-1, to=2-1]
			\arrow["g", from=1-2, to=1-3]
			\arrow["{f_A}"', from=1-2, to=2-2]
			\arrow["\Delta", from=1-3, to=1-4]
			\arrow["f"', from=1-3, to=2-3]
			\arrow["\varphi", from=1-4, to=1-5]
			\arrow["{f\otimes_Rf}"', from=1-4, to=2-4]
			\arrow["{\overline{\scriptsize{f}}}"{pos=0.3}, shift left, curve={height=-30pt}, dashed, from=1-4, to=3-4]
			\arrow["{f\otimes_Af}", from=1-5, to=2-5]
			\arrow["{\tau'}"', from=2-1, to=2-2]
			\arrow["{\textrm{id}_{R'}}"', from=2-1, to=3-1]
			\arrow["{g'}"', from=2-2, to=2-3]
			\arrow["{\textrm{id}_{A'}}"', from=2-2, to=3-2]
			\arrow["{\tilde{\Delta}}"', from=2-3, to=2-4]
			\arrow["{\textrm{id}_{B'}}"', from=2-3, to=3-3]
			\arrow[from=2-4, to=2-5]
			\arrow["p"', from=2-4, to=3-4]
			\arrow["{p_A}", from=2-5, to=3-5]
			\arrow["{\tau'}"', from=3-1, to=3-2]
			\arrow["{g'}"', from=3-2, to=3-3]
			\arrow["{\Delta'}"', from=3-3, to=3-4]
			\arrow["{\varphi'}"', from=3-4, to=3-5]
		\end{tikzcd}\eqno{(\star)}\]
		
		We know that $\ker p=\ideal{\mbox{Im}(\tilde{\Delta}\circ g'\circ \tau')}$. If $\sigma:R'\ten_RR'\rightarrow B'\ten_R B'$ is given by $\sigma:=(g'\ten_Rg')\circ(\tau'\ten_R\tau')$, it is easy to see that $$\tilde{\Delta}\circ g'\circ\tau'(r')=\sigma(r'\ten_R1_{R'}-1_{R'}\ten_Rr'), \forall r'\in R'.$$
		
		Thus, $\ker p=\ideal{\sigma(\ker\gamma)}=\ideal{e}$, where $e:=\sigma(e')$. Notice that $e^2=e$. Since $\ker p\sub\mbox{Im}(f\ten_Rf)$ then there exists $e_0\in B\ten_RB$ such that $e=(f\ten_Rf)(e_0)$. Since $e\in\ker p$ then $p((f\ten_Rf)(e_0))=p(e)=0_{B'\ten_{R'}B'}$, i.e.,  $\overline{f}(e_0)=0_{B'\ten_{R'}B'}\implies e_0\in\ker\overline{f}$.
		
		We already have $A^*_{B,R}\sub f^{-1}((A')^*_{B',R'})$. Conversely, let $x\in f^{-1}((A')^*_{B',R'})$. Thus, $f(x)\in (A')^*_{B',R'}$, i.e., $\overline{f}(\Delta(x))=\Delta'(f(x))\in\overline{\ker\varphi'}$. Once we are working with a Lipman diagram, we have $\overline{\ker\varphi'}=\overline{\ker\varphi(B'\ten_{R'}B')}$. Since $f\ten_Rf$ and $p$ are surjective then $\overline{f}=p\circ(f\ten_Rf)$ is surjective, hence, integral. Once $\overline{f}(\Delta(x))\in\overline{\ker\varphi(B'\ten_{R'}B')}$ then $\Delta(x)\in\overline{f}^{-1}(\overline{\ker\varphi(B'\ten_{R'}B')})$ and Lemma \ref{202410280018} ensures there exist $a_i\in(\ker\varphi)^i, i\in\{1,\hdots, n\}$ such that $$a_n+a_{n-1}\Delta(x)+\cdots+a_1(\Delta(x))^{n-1}+(\Delta(x))^n\in\ker\overline{f}.$$
		
		If we denote $d:=-(a_n+a_{n-1}\Delta(x)+\cdots+a_1(\Delta(x))^{n-1})$, the last equation implies that $$(\Delta(x))^n-d\in\ker\overline{f}\implies (f\ten_Rf)((\Delta(x))^n-d)\in\ker p=\ideal{f\ten_Rf(e_0)},$$
		
		\noindent and using that $f\ten_Rf$ is surjective, there exists $w\in B\ten_RB$ such that $$(f\ten_Rf)((\Delta(x))^n-d)=(f\ten_Rf)(we_0),$$\noindent  and consequently, $(\Delta(x))^n-d-we_0\in\ker(f\ten_Rf)$. If we take $J_1:=\ker(f\ten_Rf)+\ideal{e_0}$, one has $(\Delta(x))^n-d\in J_1$. Hence, $$0_{\frac{B\ten_RB}{J_1}}=(-d+J_1)+((\Delta(x))^n+J_1)$$$$=(a_n+J_1)+(a_{n-1}+J_1)(\Delta(x)+J_1)+\cdots+(a_1+J_1)(\Delta(x)+J_1)^{n-1}+(\Delta(x)+J_1)^n,$$
		
		\noindent where $a_i+J_1\in\left(\ker\varphi\dfrac{B\ten_RB}{J_1}\right)^i, \forall i\in\{1,\hdots,n\}$. Thus, $$\Delta(x)+J_1\in\overline{\left(\ker\varphi\dfrac{B\ten_RB}{J_1}\right)}. \eqno{(1)}$$
		
		We have seen that $\overline{f}(e_0)=0_{B'\ten_{R'}B'}$, and this implies $\varphi'(\overline{f}(e_0))=0_{B'\ten_{A'}B'}$. Since $(\star)$ is a commutative diagram then $p_A((f\ten_Af)(\varphi(e_0)))=0_{B'\ten_{A'}B'}$. Since $f_A$ is radicial, Lemma \ref{202410312042} guarantees that $\ker p_A\sub\sqrt{\ideal{0_{B'\ten_AB'}}}$. Thus, $$(f\ten_Af)(\varphi(e_0))\in\ker p_A\sub \sqrt{\ideal{0_{B'\ten_AB'}}}\implies \varphi(e_0)\in\sqrt{\ker(f\ten_Af)}.$$
		
		Since $f$ is an integral morphism (once it is surjective), Lemma \ref{202410241420} (a) guarantees that $$\ker(f\ten_Rf)\sub\sqrt{\ideal{0_{B\ten_{R}B}}}\mbox{ and }\ker(f\ten_Af)\sub\sqrt{\ideal{0_{B\ten_{A}B}}}.$$
		
		Hence, $\varphi(e_0)\in\sqrt{\ideal{0_{B\ten_AB}}}$ and, consequently, there exists $t\in\bN$ such that $\varphi(e_0^t)=0_{B\ten_AB}$, i.e., $e_0^t\in\ker\varphi$. Thus, $c:=\Delta(x)e_0^t\in\ker\varphi$. Denote $\mathbf{1}:=1_{B\ten_RB}$. If $J_2:=\ideal{\mathbf{1}-e_0}$, we have $$\Delta(x)-c=\Delta(x)(\mathbf{1}-e_0^t)=\Delta(x)(\mathbf{1}-e_0)(\mathbf{1}+e_0+\cdots+e_0^{t-1})\in J_2.$$ Consequently, $$\Delta(x)+J_2=\underbrace{c}_{\in\ker\varphi}+J_2\in \ker\varphi\dfrac{B\ten_RB}{J_2}\sub\overline{\left(\ker\varphi\dfrac{B\ten_RB}{J_2}\right)}$$ $$\implies \Delta(x)+J_2\in \overline{\left(\ker\varphi\dfrac{B\ten_RB}{J_2}\right)}. \eqno{(2)}$$
		
		Moreover, since $e^2=e$ and $e=(f\ten_Rf)(e_0)$ then $e_0-e_0^2\in\ker(f\ten_Rf)$. Thus, $$J_1J_2=(\ker(f\ten_Rf)+\ideal{e_0})\ideal{\mathbf{1}-e_0}=\ker(f\ten_Rf)\ideal{\mathbf{1}-e_0}+\ideal{e_0}\ideal{\mathbf{1}-e_0}$$$$\sub \ker(f\ten_Rf)+\underbrace{\ideal{e_0-e_0^2}}_{\sub \ker(f\ten_Rf)}\sub \ker(f\ten_Rf)\sub \sqrt{\ideal{0_{B\ten_RB}}}.$$
		
		Hence, $J_1J_2$ is a nil ideal of $B\ten_RB$, and using (1), (2) and Lemma \ref{202410262229}, we can conclude that $\Delta(x)\in \overline{\ker\varphi}$. Therefore, $x\in A^*_{B,R}$.
	\end{proof}
	
	As a consequence, we obtain the \textit{unramified part} of Proposition \ref{202410262346} as a corollary of our theorem.
	
	\begin{cor}
		Consider the sequence of ring morphisms % https://q.uiver.app/#q=WzAsNCxbMSwwLCJSJyJdLFsyLDAsIkEiXSxbMywwLCJCIl0sWzAsMCwiUiJdLFswLDEsIlxcdGF1JyJdLFsxLDIsImciXSxbMywwLCJmX1IiXSxbMywxLCJcXHRhdSIsMix7ImN1cnZlIjoyfV1d
		\[\begin{tikzcd}
			R & {R'} & A & B
			\arrow["{f_R}", from=1-1, to=1-2]
			\arrow["\tau"', curve={height=24pt}, from=1-1, to=1-3]
			\arrow["{\tau'}", from=1-2, to=1-3]
			\arrow["g", from=1-3, to=1-4]
		\end{tikzcd}\]
		
		If $f_R$ is unramified then $A^*_{B,R}=A^*_{B,R'}$.
	\end{cor}
	
	\begin{proof}
		
		Notice that % https://q.uiver.app/#q=WzAsNixbMCwwLCJSIl0sWzAsMSwiUiciXSxbMSwwLCJBIl0sWzEsMSwiQSJdLFsyLDAsIkIiXSxbMiwxLCJCIl0sWzAsMiwiXFx0YXUiXSxbMiw0LCJnIl0sWzQsNSwiXFx0ZXh0cm17aWR9X0IiXSxbMiwzLCJcXHRleHRybXtpZH1fQSJdLFswLDEsImZfUiIsMl0sWzEsMywiXFx0YXUnIiwyXSxbMyw1LCJnIiwyXV0=
		\[\begin{tikzcd}
			R & A & B \\
			{R'} & A & B
			\arrow["\tau", from=1-1, to=1-2]
			\arrow["{f_R}"', from=1-1, to=2-1]
			\arrow["g", from=1-2, to=1-3]
			\arrow["{\textrm{id}_A}", from=1-2, to=2-2]
			\arrow["{\textrm{id}_B}", from=1-3, to=2-3]
			\arrow["{\tau'}"', from=2-1, to=2-2]
			\arrow["g"', from=2-2, to=2-3]
		\end{tikzcd}\] is a Lipman diagram, where $f_R$ is unramified, $\id_A$ is radicial and $\id_B$ is surjective. Thus, the result is a direct consequence of Theorem \ref{202410280007}.
	\end{proof}

	\section*{Acknowledgements}

	Thiago da Silva is funded by CAPES grant number 88887.909401/2023-00 and CAPES grant number 88887.897201/2023-00.

	\vspace{0.5cm}
	
	\textsc{Thiago da Silva (Federal University of Espírito Santo)}
	
	thiago.silva@ufes.br
	

\begin{thebibliography}{99}
		
		\bibitem{adkins1} W. Adkins, A. Andreotti and J. Leahy, \textit{Weakly normal complex spaces,} Lincei Contributi del Cen. Interdisciplinare di Scienze Mat. 55, (1981).
		
		\bibitem{adkins} W. Adkins, \textit{Weak normality and Lipschitz saturation for ordinary singularities}, Compositio Mathematica, Vol. 51, no. 2, 149-157, (1984).
		
		\bibitem{AB} A. Andreotti and E. Bombieri, \textit{Sugli omeomorfismi dell variettà algebriche}. Ann. Scuola Norm. Sup. Pisa 23, 430-450, (1969).
		
		%\bibitem{B} L. Birbrair, {\it Local bi-Lipschitz classification of 2-dimensional semialgebraic sets.} Houston J. Math. 25, no. 3, 453-472, (1999).
		
		%\bibitem{G1} T. Gaffney, {\it The genericity of the infinitesimal Lipschitz condition for hypersurfaces}, J. Singul. 10, 108-123, (2014). 
		
		%\bibitem{SG} T. Gaffney and T. da Silva, \textit{Infinitesimal Lipschitz conditions on a family of analytic varieties}, arXiv: 1902.03194 [math.AG]
		
		
		%\bibitem{G2} T. Gaffney, {\it Bi-Lipschitz equivalence, integral closure and invariants,} Proceedings of the 10th International Workshop on Real and Complex Singularities. Edited by: M. Manoel, Universidade de S\~ao Paulo, M. C. Romero Fuster, Universitat de Valencia, Spain, C. T. C. Wall, University of Liverpool, London Mathematical Society Lecture Note Series (No.380) November (2010).
		
		%\bibitem{G3} T. Gaffney, {\it Integral Closure of Modules and Whitney equisingularity,} Invent. Math. 107, 301-322, (1992).
		
		%\bibitem{G4} T. Gaffney, {\it Infinitesimal Bi-Lipschitz Equivalence of Functions}, arxiv: 1601.05147v1 [math.CV]  
		
		%\bibitem{G5} T. Gaffney, \textit{Generalized Buchsbaum-Rim Multiplicities and a Theorem of Rees}, Communications in Algebra, 31, 3811-3828, (2003).
		
		%\bibitem{G6} T. Gaffney, \textit{Equisingularity of Plane Sections, $t^1$ Condition, and the Integral Closure of Modules}, Real and Complex Singularities, ed. by W. L. Marar, Pitman Research Series in Math. 333, Longman, (1995).
		
		%\bibitem{G7} T. Gaffney, \textit{The Multiplicity Polar Theorem}, arxiv:math/0703650v1 [math.CV], (2007).
		
		%\bibitem{G8} T. Gaffney, \textit{The Multiplicity of Pairs of Modules and Hypersurface Singularities}, Proceedings of the 7th International Workshop on Real and Complex Singularities. Trends in Mathematics, (2007).
		
		%\bibitem{GK} T. Gaffney and S. Kleiman, {\it Specialization of integral dependence for modules,}, Inventiones mathematicae. 137, 541-574 (1999).
		
		%\bibitem{atiyah} M. F. Atiyah and L. G. Macdonald, \textit{Introduction to Commutative Algebra}. Addison-Wesley Publishing Company, (1969).
		
		\bibitem{kleiman} A. Altman and S. Kleiman, \textit{A term of commutative algebra}, Cambridge, USA, (2021).
		
		%\bibitem{bernard} F. Bernard, \textit{Relative Lipschitz saturation of complex algebraic varieties}, arXiv:2312.14517, (2023).
		
		%\bibitem{bourbaki} N. Bourbaki, \textit{Algebra. II. } Chapters 4–7. 1990. 
		
		%	\bibitem{Silva1} T. da Silva, \textit{Categorical aspects of Gaffney’s double structure of a module}, Matemática Contemporânea, Vol. 53, 213-232, (2023).
		
		\bibitem{SR} T. da Silva and M. Ribeiro, \textit{Universally injective and integral contractions on relative Lipschitz saturation of algebras}, Journal of Algebra, Vol. 662, 902-922, (2024).  %https://doi.org/10.1016/j.jalgebra.2024.08.024.
		
		%\bibitem{SGP} T. da Silva and N. G. Grulha Jr. and M. Pereira, \textit{The Bi-Lipschitz Equisingularity of Essentially Isolated Determinantal Singularities}, Bulletin of the Brazilian Mathematical Society, New Series, 49, 637-645, (2018).
		
		%\bibitem{SGP2} T. da Silva and N. G. Grulha Jr. and M. Pereira, \textit{Real and Complex Integral Closure, Lipschitz equisingularity and applications on square matrices}, Journal of Singularities, Vol. 22, 215-226, (2020).		
		
		%\bibitem{GAP2} M. Delgado, P. A. García-Sánchez and J.J. Morais, \textit{numericalsgps -- A package for numerical semigroups}, Version 1.3.1, Centro de Matemática da Universidade do Porto, Portugal and Universidad de Granada, Spain \url{https://docs.gap-system.org/pkg/numericalsgps/doc/chap0.html}. (2005--2015).
		
		%	\bibitem{FR2} A. Fernandes and M. A. S. Ruas, \textit{Bilipschitz determinacy of quasihomogeneous germs}, Glasgow Mathematical Journal, 46 (1), 77-82 (2004).
		
		%\bibitem{FR} A. Fernandes and M. A. S. Ruas, \textit{Rigidity of bi-Lipschitz equivalence of weighted homogeneous function-germs in the plane}, Proc. Amer. Math. Soc. 141, 1125-1133 (2013).
		
		%	\bibitem{GAP} The GAP$\sim$Group, \textit{GAP -- Groups, Algorithms, and Programming, Version 4.13.1};  \url{https://www.gap-system.org}, (2024).			
		
		%\bibitem{gaffney1} T. Gaffney, \textit{Bi-Lipschitz equivalence, integral closure and invariants,} Real and complex singularities, 125137, London Math. Soc. Lecture Note Ser., 380, Cambridge Univ. Press, Cambridge, (2010).
		
		%\bibitem{gaffney2} T. Gaffney, \textit{The genericity of the infinitesimal Lipschitz condition for hypersurfaces}, Journal of Singularities, Vol. 10, 108-123, (2015).
		
		%\bibitem{GS} T. Gaffney and T. da Silva, \textit{Infinitesimal Lipschitz conditions on a family of analytic varieties: genericity and necessity}, São Paulo J. Math. Sci., https://doi.org/10.1007/s40863-024-00452-5, (2024).
		
		%\bibitem{GS2} T. Gaffney and T. da Silva, \textit{The generic equivalence among the Lipschitz saturations of a sheaf of modules}, Research in the Mathematical Sciences, Vol. 11, (32), (2024).
		
		\bibitem{GT} S. Greco and C. Traverso, \textit{On seminormal schemes}, Comp. Math. 40, 325-365 (1980).
		
		\bibitem{EGA} A. Grothendieck and J. Dieudonné, \textit{Eléments de géométrie algébrique}, Springer-Verlag, Berlin, (1971).
		
		%\bibitem{KT} S. Kleiman and A. Thorup, {\it A geometric theory of the Buchsbaum-Rim \newline multiplicity}, J. Algebra 167, 168-231, (1994).
		
		%\bibitem{LT} M. Lejeune-Jalabert and B. Teissier, {\it Cl\^oture int\'egrale des id\'eaux et equisingularit\'e}, S\'eminaire Lejeune-Teissier, Centre de Math\'ematiques \'Ecole Polytechnique, (1974) Publ. Inst. Fourier St. Martin d'Heres, F-38402, arxiv:0803.2369 (2008).
		
		\bibitem{L} J. Lipman, {\it Relative Lipschitz-saturation}, Amer. J. Math. 97, no. 3, 791-813, (1975).
		
		%\bibitem{OSCAR} OSCAR -- Open Source Computer Algebra Research system, Version 1.0.0, The OSCAR Team, (https://www.oscar-system.org), (2024).
		
		%\bibitem{L} Looijenga, E. J. N.: {\it Isolated singular points on complete intersection} (Lond. Math. Soc. Lect. Note Ser., vol 77) Cambridge: Cambridge University Press, (1984).
		
		
		
		\bibitem{Ma1} M. Maranesi, \textit{Una caratterizzazione della seminormalizzazione}, Boll. Un. Mat. Ital., (5), 15-A, 205-213 (1978).
		
		\bibitem{Ma} M. Maranesi, \textit{Some properties of weakly normal varieties}, Nagoya Math. J. 77, 61-74, (1980).
		
		%\bibitem{M1} T. Mostowski, {\it A criterion for Lipschitz equisingularity.} Bull. Polish Acad. Sci. Math. 37 (1989), no. 1-6, 109-116 (1990).
		
		%\bibitem{M2} T. Mostowski, {\it Tangent cones and Lipschitz stratifications.} Singularities (Warsaw, 1985), 303-322, Banach Center Publ., 20, PWN, Warsaw, (1988).
		
		%\bibitem{PA1} A. Parusinski, {\it Lipschitz stratification of real analytic sets.} Singularities (Warsaw, 1985), 323-333, Banach Center Publ., 20, PWN, Warsaw, (1988). 
		
		%\bibitem{PA2} A. Parusinski, {\it Lipschitz properties of semi-analytic sets.} Ann. Inst. Fourier (Grenoble) 38, no. 4, 189-213, (1988).
		
		\bibitem{PT} F. Pham and B. Teissier, {\it Fractions lipschitziennes d'une alg\'ebre analytique complexe et saturation de Zariski}, Centre de Math\'ematiques de l'Ecole Polytechnique (Paris), http://people.math.jussieu.fr/teissier/old-papers.html, (1969).
		
		\bibitem{P1} F. Pham, {\it Fractions lipschitziennes et saturation de Zariski des alg\`ebres analytiques complexes.} Expos\'e d'un travail fait avec Bernard Teissier. Fractions lipschitziennes d'une alg\`ebre analytique complexe et saturation de Zariski, Centre Math. l\'Ecole Polytech., Paris, 1969. Actes du Congr\`es International des Math\'ematiciens (Nice, 1970), Tome 2, pp. 649-654. Gauthier-Villars, Paris, (1971).
		
		%\bibitem{stacks} Stacks Project, Schemes, \href{https://stacks.math.columbia.edu/tag/01S2}{https://stacks.math.columbia.edu/tag/01S2}, New York, USA.
		
		%\bibitem{RGS} J.C. Rosales and P.A. Garcia-Sánchez, {\it Numerical Semigroups}, (First. ed.). New York: Springer, (2009).
		
		%\bibitem{SH} I. Swanson and C. Huneke, \textit{Integral closure of ideals, rings and modules,} London Mathematical Society, Lecture Notes Series 336, Cambridge, UK, (2006).
		
		%\bibitem{Syl} J. J. Sylvester, \textit{On subvariants, i.e., semi-invariants to binary quantities of an unlimited order,} Amer. J. Math. 5, 119-136, (1882).
		
		%\bibitem{zariski1965} O. Zariski, {\it Studies in equisingularity i. equivalent singularities of plane algebroid curves}, American Journal of Mathematics, Vol. 87, No. 2, pp. 507-636, (1965).
		
		%\bibitem{zariski1965-2} O. Zariski, {\it Studies in equisingularity ii. equisingularity in codimension 1 (and characteristic zero) }, American Journal of Mathematics, Vol. 87, No. 4, pp. 972-1006, (1965).
		
		%\bibitem{Z} O. Zariski, {\it General Theory of Saturation and of Saturated Local Rings. II. Saturated local rings of dimension 1} Amer. J. Math. 93, 872-964, (1971).
		
		
		
	\end{thebibliography}
\end{document}